\documentclass[11pt]{article}
\usepackage[utf8]{inputenc}
\usepackage[top=2cm,bottom=2cm,left=2cm,right=2cm]{geometry}
\usepackage{amsthm,amsmath,amssymb}
\usepackage{hyperref} 
\usepackage{enumerate}
\usepackage{cleveref}

\usepackage{authblk} 
\usepackage{bm} 
\usepackage{mathtools} 
\usepackage{graphicx} 
\usepackage[symbol]{footmisc} 
\usepackage{multirow,makecell}
\usepackage{tikz}
\usetikzlibrary{arrows.meta}
\usetikzlibrary{positioning}

\usepackage{nicematrix}
\usepackage{booktabs}
\usepackage{colortbl}

\newcommand{\doublehat}[1]{%
\overline{#1}}

\usepackage{caption}
\usepackage{xcolor}
\newcommand{\blue}[1]{\textcolor{blue}{#1}}
\newcommand{\orange}[1]{\textcolor{orange}{#1}}
\newcommand{\purple}[1]{\textcolor{purple}{#1}}
\definecolor{mygreen}{rgb}{0,0.75,0}
\definecolor{mygray}{rgb}{0.75,0.75,0.75}
\newcommand{\mygreen}[1]{\textcolor{mygreen}{#1}}
\newcommand{\bbllt}{\blue{$0$}}
\newcommand{\obllt}{\orange{$+$}}
\newcommand{\gbllt}{\mygreen{$-$}}
\newcommand{\pbllt}{\purple{B}}
\definecolor{fillbg2}{rgb}{0.85, 1, 1}

\newcommand{\msf}{\mathsf}
\newcommand{\X}{\mathsf{X}}
\newcommand{\Y}{\mathsf{Y}}
\newcommand{\Z}{\mathsf{Z}}
\newcommand{\0}{\mathsf{0}}

\newtheorem{theorem}{Theorem}
\newtheorem{lemma}[theorem]{Lemma}
\newtheorem{corollary}[theorem]{Corollary}

\theoremstyle{definition}
\newtheorem{remark}[theorem]{Remark}
\newtheorem{example}[theorem]{Example}

\DeclareMathOperator{\tr}{tr}

\title{Bifurcations in planar, quadratic mass-action networks\\ with few reactions and low molecularity}
\author[1]{Murad Banaji\thanks{Corresponding author.}}
\author[2]{Bal\'azs Boros\thanks{This research was funded in whole or in part by the Austrian Science Fund (FWF) 10.55776/P32532. For open access purposes, the author has applied a CC BY public copyright license to any accepted manuscript version arising from this submission.}}
\author[2]{Josef Hofbauer}
\affil[1]{Mathematical Institute, University of Oxford, United Kingdom}
\affil[2]{Department of Mathematics, University of Vienna, Austria}
\date{}

\begin{document}

\maketitle

\begin{abstract}
In this paper we study bifurcations in mass-action networks with two chemical species and reactant complexes of molecularity no more than two. We refer to these as planar, quadratic networks as they give rise to (at most) quadratic differential equations on the nonnegative quadrant of the plane. Our aim is to study bifurcations in networks in this class with the fewest possible reactions, and the lowest possible product molecularity. We fully characterise generic bifurcations of positive equilibria in such networks with up to four reactions, and product molecularity no higher than three. In these networks we find fold, Andronov--Hopf, Bogdanov--Takens and Bautin bifurcations, and prove the non-occurrence of any other generic bifurcations of positive equilibria. In addition, we present a number of results which go beyond planar, quadratic networks. For example, we show that mass-action networks without conservation laws admit no bifurcations of codimension greater than $m-2$, where $m$ is the number of reactions; we fully characterise quadratic, rank-one mass-action networks admitting fold bifurcations; and we write down some necessary conditions for Andronov--Hopf and cusp bifurcations in mass-action networks. Finally, we draw connections with a number of previous results in the literature on nontrivial dynamics, bifurcations, and inheritance in mass-action networks.  
\end{abstract}


\section{Introduction}

Although some of our results apply more widely, the main focus of this paper is on bifurcations in planar, quadratic mass-action networks, namely networks involving two chemical species, and with at most bimolecular reactant complexes. Mathematically, such networks give rise to quadratic differential equations on the nonnegative quadrant of $\mathbb{R}^2$. 
This paper takes a step towards enumerating planar, quadratic networks which admit local bifurcations of equilibria, and are minimal in the sense of having the fewest reactions and the lowest product molecularity. 

We fully characterise {\em all} planar, quadratic, trimolecular mass-action networks with no more than four reactions admitting generic bifurcations of equilibria. These networks exhibit {\em fold}, {\em Andronov--Hopf}, {\em Bogdanov--Takens}, and {\em Bautin bifurcations} \cite{kuznetsov:2023}, namely, all generic co-dimension $1$ and $2$ bifurcations of equilibria possible in planar networks with the exception of {\em cusp} bifurcations. There is a unique network in this class admitting a Bautin bifurcation; but fold, Andronov--Hopf and Bogdanov--Takens bifurcations occur with greater frequency. On the other hand, we show that cusp bifurcation in a planar, quadratic mass-action network requires at least five reactions (see \Cref{thm:cusp}), and characterising minimal networks exhibiting cusp bifurcation is deferred to future work. 

Let us explain the emphasis on planar, quadratic, trimolecular networks with four reactions. It has been shown previously that planar {\em bimolecular} mass-action networks forbid limit cycles \cite{pota:1983}. Therefore quadratic, trimolecular networks are the simplest planar networks where we can hope to find limit cycles; however, in recent work \cite[Theorem 1]{banaji:boros:hofbauer:2024a}, we showed that networks in this class with no more than three reactions forbid limit cycles. Hence, in seeking bifurcations involving limit cycles we must focus on networks with four or more reactions. In fact, there are 198 dynamically nonequivalent four-reaction networks in this class admitting Andronov--Hopf bifurcation (see \Cref{thm:A-H}). Four of these admit both sub- and supercritical Andronov--Hopf bifurcations but only one admits a Bautin bifurcation leading to two limit cycles.

Unlike Andronov--Hopf bifurcations, fold bifurcations can occur in rank-one networks. In fact, remarkably, quadratic, rank-one, networks admitting fold bifurcation can be fully characterised, and all are closely related to a special one-species, three-reaction network (see \Cref{thm:quadratic_rank-one_fold}). While the rank-two case is more complicated, it follows from \cite[Lemma 3.1]{banaji:boros:2023} that any planar, rank-two network admitting a fold bifurcation must have at least four reactions. We are thus led, again, to planar networks with four reactions. We find that there are 831 dynamically nonequivalent planar, quadratic, trimolecular networks with four reactions admitting a nondegenerate fold bifurcation (\Cref{thm:834_quadratic_trimolecular_fold}). These include the simplest {\em bimolecular} networks admitting a fold bifurcation (\Cref{thm:30_bimolecular_fold}). The enumeration also allows us to characterize the simplest {\em bistable} bimolecular networks (\Cref{lem:bistable}) with two asymptotically stable equilibria, one positive and one at the origin, thus extending and correcting the work of Wilhelm \cite{wilhelm:2009}.

Having enumerated planar, quadratic, trimolecular, four-reaction networks exhibiting fold and Andronov--Hopf bifurcations respectively, it is natural to examine networks in the intersection of these two sets. In doing so, we find 31 dynamically nonequivalent networks admitting a nondegenerate Bogdanov--Takens bifurcation \cite[Section 8.4]{kuznetsov:2023}, and another two admitting a degenerate form of the bifurcation which we refer to as a {\em vertical Bogdanov--Takens bifurcation} (see \Cref{thm:BTmain}). Other than Andronov--Hopf and fold bifurcations of equilibria, nondegenerate Bogdanov--Takens bifurcations are also associated with homoclinic bifurcations \cite[Section 6.2]{kuznetsov:2023}. Examples of reaction systems displaying Bogdanov--Takens or homoclinic bifurcations appear in \cite{hofbauer:schuster:1984}, \cite{kay:1990}, \cite{li:liu:1992}, \cite{olsen:epstein:1993}, \cite{clarke:jiang:1993}, \cite{plesa:vejchodsky:erban:2016}, \cite{delgado:HM:PL:2017}, \cite{kreusser:rendall:2021}, \cite{banaji:boros:hofbauer:2022}, \cite{plesa:2024}.

Our work is probably the first to systematically find small reaction networks admitting Bogdanov--Takens bifurcations. In fact, this work was originally motivated by examining the network
\begin{center}
\begin{tabular}{lll}
$\X \to 2\X \to 3\X$ & $\X+\Y \to 2\Y$ & $\Y \rightleftarrows \0$
\end{tabular}
\end{center}
studied by Frank-Kamenetsky and Salnikov \cite{frank-kamenetsky:salnikov:1943} in 1943. This network is known to admit a supercritical Andronov--Hopf bifurcation, and we found that it also admits a supercritical Bogdanov--Takens bifurcation (to the best of our knowledge this has not appeared in the previous literature). The network is a planar, quadratic, trimolecular network with five chemical reactions, raising the question of whether Bogdanov--Takens bifurcation might be possible in such networks with fewer reactions. Our results answer in the affirmative: as noted above, there are 31 dynamically nonequivalent, planar, quadratic, trimolecular, {\em four-reaction} networks admitting nondegenerate Bogdanov--Takens bifurcations. In fact, the five-reaction network studied by Frank-Kamenetsky and Salnikov \emph{inherits} the Bogdanov--Takens bifurcation from one of these four-reaction networks \cite{banaji:boros:hofbauer:2024b}. 

The rest of this paper is organised as follows. In \Cref{sec:prelim} we set out preliminary definitions and results. In \Cref{sec:fold} we fully characterise fold bifurcation in quadratic, rank-one networks of arbitrary molecularity; and in planar, quadratic networks with up to four reactions and product molecularity up to three. In \Cref{sec:A-H} we study Andronov--Hopf bifurcations in planar, quadratic, trimolecular networks with four reactions. We also observe that there is a unique network in this class admitting a Bautin bifurcation, and several networks admitting vertical Andronov--Hopf bifurcations. In \Cref{sec:B-T} we study Bogdanov--Takens bifurcation in planar, quadratic, trimolecular networks with four reactions, following on very naturally from the study of fold and Andronov--Hopf bifurcations in this class of networks. In \Cref{sec:other} we show that the bifurcations described so far encompass all the generic bifurcations of positive equilibria possible in planar, quadratic, trimolecular networks with up to four reactions. Finally, in \Cref{sec:conclusions} we present some broad conclusions, and directions for the future.

\section{Preliminaries}
\label{sec:prelim}

\subsection{Basic definitions}

Basic notation and terminology follow \cite{banaji:boros:2023} and \cite{banaji:boros:hofbauer:2024a}, and are outlined only briefly. We denote the nonnegative orthant in $\mathbb{R}^n$ (resp., $\mathbb{Z}^n$) by $\mathbb{R}^n_{\geq 0}$ (resp., $\mathbb{Z}^n_{\geq 0}$), and the positive orthant in $\mathbb{R}^n$ by $\mathbb{R}^n_+$. The symbol ``$\circ$'' is used to denote the entrywise product of vectors. Functions, particularly the natural logarithm $\ln(\cdot)$, will often be applied to vectors entrywise. Given $x := (x_1, \ldots, x_n)^\mathrm{t}$ and $a:= (a_1, \ldots, a_n)$, we write $x^a:= x_1^{a_1}x_2^{a_2}\cdots x_n^{a_n}$. If $A$ is an $m \times n$ matrix with rows $A_1, \ldots, A_m$, we write $x^A:= (x^{A_1}, \ldots, x^{A_m})^\mathrm{t}$. With these conventions, $(x^A)^B = x^{BA}$, and $\ln(x^A) = A\,\ln x$. We write $\bm{1}$ to denote a vector of ones of length to be inferred from the context.

\subsubsection*{Chemical reaction networks} A {\em complex} is a formal linear combination with nonnegative integer coefficients on some set of (chemical) {\em species}. The coefficient of a species in a complex is its {\em stoichiometry}. The {\em molecularity} of a complex is the sum of the stoichiometries of species in this complex. We refer to a complex as {\em bimolecular} if it has molecularity at most two, and {\em trimolecular} if it has molecularity at most three. A (chemical) {\em reaction} on this set of species is an ordered pair of complexes termed the {\em reactant complex} and {\em product complex}, where we always assume that the two complexes are different. A chemical reaction is {\em quadratic} if its reactant complex is bimolecular. It is {\em bimolecular} (resp., {\em trimolecular}) if all of its complexes are bimolecular (resp., trimolecular). A chemical reaction network (CRN, or network for short) is defined as a set of reactions on a set of species. It is {\em quadratic} (resp., {\em bimolecular}, resp., {\em trimolecular}) if all of its reactions are quadratic (resp., bimolecular, resp., trimolecular).

Consider a CRN with $n$ species and $m$ reactions. We may assume at the outset a fixed ordering on the species and reactions, in which case the network is associated with various $n \times m$ matrices: the {\em reactant matrix} $\Gamma_\ell$, where $(\Gamma_\ell)_{i,j}$ is the stoichiometry of the $i$th species in the reactant complex of the $j$th reaction; the {\em product matrix} $\Gamma_r$, where $(\Gamma_r)_{i,j}$ is the stoichiometry of the $i$th species in the product complex of the $j$th reaction; and the {\em stoichiometric matrix} $\Gamma := \Gamma_r-\Gamma_\ell$. The columns of $\Gamma$ are the {\em reaction vectors} of the associated reactions, the span of the reaction vectors is called the \emph{stoichiometric subspace}, and the rank of $\Gamma$ is termed the {\em rank} of the network. Intersections of cosets of $\mathrm{im}\,\Gamma$ with $\mathbb{R}^n_{\geq 0}$ are termed {\em stoichiometric classes}; and intersections of cosets of $\mathrm{im}\,\Gamma$ with $\mathbb{R}^n_{+}$ are termed {\em positive stoichiometric classes}. We refer to a CRN with $n$ chemical species, $m$ chemical reactions, and rank $r$, as an $(n,m,r)$ network.

\subsubsection*{Mass-action networks} An $(n,m,r)$ CRN with mass-action kinetics, or a {\em mass-action network} for short, can be defined by its stoichiometric matrix $\Gamma$, its reactant matrix $\Gamma_\ell$, and a vector of rate constants $\kappa \in \mathbb{R}^m_+$. It gives rise to the ODE system
\[
\dot x = \Gamma(\kappa \circ x^A)
\]
where $x \in \mathbb{R}^n_{\geq 0}$ is the vector of species concentrations, and $A:= \Gamma_\ell^\mathrm{t}$. We refer to the $i$th chemical species as {\em trivial} if each entry in the $i$th row of $\Gamma$ is zero. In this case, $\dot{x}_i\equiv 0$, and therefore $x_i$ is constant over time. (In fact this holds true for any choice of reaction rates, not necessarily mass action.) 

We say that a mass-action network {\em admits} some dynamical behaviour if this behaviour occurs for some choice of rate constants; and it admits some bifurcation if this bifurcation occurs on some stoichiometric class as we vary its rate constants. A necessary and sufficient condition for a mass-action network to admit positive limit sets is that it must be {\em dynamically nontrivial}, namely $\mathrm{ker}_+\,\Gamma:=\mathrm{ker}\,\Gamma \cap \mathbb{R}^m_{+}$ is nonempty (see the discussion in \cite{banaji:boros:hofbauer:2024a}). Clearly a dynamically nontrivial $(n,m,r)$ network must have $m > r$.

\subsection{Isomorphism and equivalence of networks}
\label{sec:isomorph}
Two CRNs are {\em isomorphic} if some reordering/renaming of the species and reactions of one network gives rise to the other. Weaker notions of equivalence are possible amongst mass-action networks. For example, two networks on the same set of species which give rise to the same set of mass-action differential equations (perhaps after permuting species and reactions), are termed {\em unconditionally confoundable}, or {\em dynamically equivalent} \cite{craciun:pantea:2008}. Note that this is a rather restrictive condition, allowing changes of parameters, but not coordinates. Whenever a mass-action network is dynamically equivalent to a network with fewer reactions, i.e., some reaction vector of the network lies in the positive span of other reaction vectors on the same complex, we say that the network includes {\em redundant reactions}.

Here we are concerned with equivalences amongst mass-action networks with the same reactant matrix $\Gamma_\ell$ (perhaps after permuting species and reactions). Consider two such networks with $n$ species and $m$ reactions giving rise to ODE systems
\begin{equation}
\label{eq:sameA}
\dot x = \Gamma(\kappa \circ x^A) =:f(x, \kappa) \quad \mbox{and} \quad \dot y = \widehat{\Gamma}(\hat{\kappa}\circ y^A) =:g(y,\hat{\kappa})
\end{equation}
respectively. We refer to these networks as {\em smoothly equivalent} if a smooth reparameterisation and recoordinatisation takes trajectories of one to those of the other. More precisely, the networks are smoothly equivalent if we can find a smooth diffeomorphism $G \colon \mathbb{R}^m_+ \to \mathbb{R}^m_{+}$ and a smooth (parameter-dependent) recoordinatisation $F \colon \mathbb{R}^n_{\geq 0} \times \mathbb{R}^m_+ \to \mathbb{R}^n_{\geq 0}$ such that $F_\kappa(\cdot) := F(\cdot, \kappa)$ is a diffeomorphism on $\mathbb{R}^n_{\geq 0}$ for each $\kappa \in \mathbb{R}^m_+$, and such that $(x,\kappa) \mapsto (F(x,\kappa),G(\kappa))$ takes trajectories of the first network to the second, i.e., 
\[
g(F(x,\kappa), G(\kappa)) = D_xF(x,\kappa) f(x,\kappa)\,.
\]

Clearly, smoothly equivalent networks are dynamically equivalent after a smooth (parameter-dependent) recoordinatisation. The simplest sufficient condition for smooth equivalence of two networks as in \eqref{eq:sameA} is if, perhaps after permuting columns of $\Gamma$ corresponding to the same reactant complex, $\widehat{\Gamma} = \Gamma\,D$, where $D$ is a positive diagonal matrix. In this case we call the two networks {\em simply equivalent}. We note that simply equivalent networks are dynamically equivalent \cite[Theorem A.3]{banaji:boros:2023}. Generalising simple equivalence, we say that the two networks in \eqref{eq:sameA} are {\em diagonally equivalent} if, perhaps after permuting columns of $\Gamma$ corresponding to the same reactant complex, $\widehat{\Gamma} = D_1\,\Gamma\,D_2$, where $D_1$ and $D_2$ are positive diagonal matrices. Diagonally equivalent networks are smoothly equivalent as we show in the next lemma. 

\begin{lemma}
Diagonally equivalent mass-action networks are smoothly equivalent. 
\end{lemma}
\begin{proof}
Consider two diagonally equivalent networks giving rise to ODE systems as in \eqref{eq:sameA} with $\widehat{\Gamma} = D_1\,\Gamma\,D_2$, where $D_1$ and $D_2$ are positive diagonal matrices. Define $d_1\in \mathbb{R}^n_+$ and $d_2\in \mathbb{R}^m_+$ by $D_1 = \mathrm{diag}(d_1)$ and $D_2 = \mathrm{diag}(d_2)$. Defining $y = d_1\circ x$ and $\hat{\kappa} = d_2^{-1}\circ \kappa \circ d_1^{-A}$, we see that
\[
\dot y = d_1\circ \dot{x} = d_1\circ \Gamma(\kappa \circ d_1^{-A}\circ y^A) = d_1\circ \Gamma(d_2\circ \hat{\kappa} \circ y^A)= \widehat{\Gamma}(\hat{\kappa} \circ y^A),
\]
confirming that the smooth, bijective, change of parameters and coordinates $(x,\kappa) \mapsto (d_1\circ x, \,d_2^{-1}\circ \kappa \circ d_1^{-A})$, takes trajectories of one system into those of the other.
\end{proof}

\begin{example}
The following two non-isomorphic mass-action networks appear in \Cref{tab:B-T} below as examples of quadratic, trimolecular $(2,4,2)$ networks admitting Bogdanov--Takens bifurcation:
\begin{center}
\begin{tabular}{*{4}{r@{\hspace{4pt}$\to$\hspace{4pt}}l}}
\arrayrulecolor{mygray}
$2\X$ & $3\X$ & $\X+\Y$ & $2\X$ & $\0$ & $\X+2\Y$ & $\X$ & $\0$ \\
\hline
\multicolumn{8}{c}{} \\[-10pt]
$2\X$ & $3\X$ & $\X+\Y$ & $3\X$ & $\0$ & $\X+\Y$ & $\X$ & $\0$ \\
\end{tabular}
\end{center}
A helpful graphical representation of each network, its {\em Euclidean embedded graph} \cite{craciun:2019}, is shown below. 
\begin{center}
\begin{tikzpicture}
\tikzset{bullet/.style={inner sep=1.5pt,outer sep=1.5pt,draw,fill,blue,circle}};
\tikzset{myarrow/.style={arrows={-stealth},very thick,blue}};
\draw [step=1, gray, very thin] (0,0) grid (3.5,2.5);
\draw [ -, black] (0,0)--(3.5,0);
\draw [ -, black] (0,0)--(0,2.5);

\node[bullet] (P1) at (2,0) {};
\node[bullet] (P2) at (3,0) {};
\node[bullet] (P3) at (1,1) {};
\node[bullet] (P4) at (1,0) {};
\node[bullet] (P5) at (0,0) {};
\node[bullet] (P6) at (1,2) {};

\node [below]       at (P1) {$2\X$};
\node [below]       at (P2) {$3\X$};
\node [above right] at (P3) {$\X+\Y$};
\node [below]       at (P4) {$\X$};
\node [below left]  at (P5) {$\0$};
\node [above right] at (P6) {$\X+2\Y$};

\draw[myarrow] (P1) to node {} (P2);
\draw[myarrow] (P3) to node {} (P1);
\draw[myarrow] (P5) to node {} (P6);
\draw[myarrow] (P4) to node {} (P5);

\begin{scope}[xshift=5cm]
\draw [step=1, gray, very thin] (0,0) grid (3.5,2.5);
\draw [ -, black] (0,0)--(3.5,0);
\draw [ -, black] (0,0)--(0,2.5);

\node[bullet] (P1) at (2,0) {};
\node[bullet] (P2) at (3,0) {};
\node[bullet] (P3) at (1,1) {};
\node[bullet] (P4) at (1,0) {};
\node[bullet] (P5) at (0,0) {};

\node [below]      at (P1) {$2\X$};
\node [below]      at (P2) {$3\X$};
\node [above left] at (P3) {$\X+\Y$};
\node [below]      at (P4) {$\X$};
\node [below left] at (P5) {$\0$};

\draw[myarrow] (P1) to node {} (P2);
\draw[myarrow] (P3) to node {} (P2);
\draw[myarrow] (P5) to node {} (P3);
\draw[myarrow] (P4) to node {} (P5);
\end{scope}

\end{tikzpicture}
\end{center}
It is easily confirmed by examining their stoichiometric matrices that the two networks are diagonally equivalent. 
\end{example}

\subsection{Subnetworks, enlargements and inheritance} \label{subsec:inheritance}

The most useful notion of ``subnetwork'' in the theory of CRNs is dependent on the application, with the caveat that the relationship of being a subnetwork should induce a partial ordering on the set of CRNs. While there is no one correct notion of subnetwork in CRNs, there are natural subnetwork relationships between CRNs, of importance for the results in this paper. 

Given CRNs $\mathcal{R}_1$ and $\mathcal{R}_2$, following the terminology introduced in \cite{banaji:pantea:2018}, we say that $\mathcal{R}_1$ is an {\em induced subnetwork} of $\mathcal{R}_2$ if $\mathcal{R}_1$ can be obtained from $\mathcal{R}_2$ by deleting species and reactions from $\mathcal{R}_2$. Note that deleting a species means removing it from every reaction in which it occurs; and if doing so results in reactions with identical reactant and product complexes, we remove these too. 

Going beyond induced subnetworks, we have previously written down a set of enlargements of mass-action networks, denoted E1--E6, which preserve their capacity for various dynamical behaviours and bifurcations \cite{banaji:2023}, \cite{banaji:boros:hofbauer:2024b}. The key results are summarised in the following lemma on the {\em inheritance} of nondegenerate behaviours, and of bifurcations, amongst mass-action networks.
\begin{lemma}
\label{lem:inherit}
Consider mass-action networks $\mathcal{R}_1$ and $\mathcal{R}_2$ with $\mathcal{R}_1$ being a subnetwork of $\mathcal{R}_2$ in the sense that $\mathcal{R}_2$ can be obtained from $\mathcal{R}_1$ via a sequence of enlargements of the form E1--E6. If $\mathcal{R}_1$ admits multiple positive nondegenerate equilibria; a nondegenerate periodic orbit; or a local bifurcation of equilibria transversely unfolded by its rate constants, then the same holds for $\mathcal{R}_2$.
\end{lemma} 
\begin{proof}
These claims and their proofs appear in \cite{banaji:2023}, \cite{banaji:boros:hofbauer:2024b}, and a number of references therein. 
\end{proof}

\begin{remark}
Neither the list of dynamical behaviours which can be inherited in \Cref{lem:inherit}, nor the list of enlargements under which the conclusions of \Cref{lem:inherit} hold, are exhaustive.
\end{remark}

Consider some $(n,m,r)$ mass-action network with a dynamical behaviour of interest. If the behaviour is not inherited from any smaller network via \Cref{lem:inherit} we refer to the network as an {\em atom} of this behaviour.

For completeness, we list the enlargements E1, E2, E3, and E6 of a CRN $\mathcal{R}$ which appear explicitly in our results or remarks to follow. For more details, the reader is referred to \cite{banaji:2023} or \cite{banaji:boros:hofbauer:2024b}.
\begin{enumerate}
\item[E1.] {\em A new linearly dependent reaction.} We add to $\mathcal{R}$ a new reaction involving only existing chemical species of $\mathcal{R}$, and in such a way that the rank of the network is preserved.
\item[E2.] {\em The fully open extension.} We add in (if absent) all chemical reactions of the form $\0 \rightarrow \X_i$ and $\X_i \rightarrow \0$ for each chemical species $\X_i$ of $\mathcal{R}$.
\item[E3.] {\em A new linearly dependent species.} We add a new chemical species into some nonempty subset of the reactions of $\mathcal{R}$, in such a way that the rank of the network is preserved.
\item[E6.] {\em Splitting reactions.} We split some reactions of $\mathcal{R}$ and insert complexes involving at least as many new species as the number of reactions split. Moreover, the new species figure nontrivially in the enlarged CRN in the sense that the submatrix of the new stoichiometric matrix corresponding to the added species has rank equal to the number of reactions which are split.
\end{enumerate}

\begin{remark}
\label{rem:induced}
If we work within a class of networks of {\em fixed rank}, then $\mathcal{R}_1$ is an induced subnetwork of $\mathcal{R}_2$ if and only if $\mathcal{R}_2$ can be obtained from $\mathcal{R}_1$ via a sequence of enlargements E1 and E3. Consequently, by \Cref{lem:inherit}, mass-action networks inherit nondegenerate dynamical behaviours and bifurcations from induced subnetworks of the same rank.
\end{remark}

\subsection{Existence and nondegeneracy of positive equilibria in mass-action networks}

For the remainder of this section, we fix the following notation for an arbitrary mass-action network:
\begin{itemize}
\item The stoichiometric matrix is denoted by $\Gamma$.
\item The (column) vector of mass-action rate constants is denoted by $\kappa$.
\item The reactant matrix is denoted by $\Gamma_\ell$, and we write $A:=\Gamma_\ell^\mathrm{t}$ for brevity.
\item $W^\mathrm{t}$ is any matrix whose columns form a basis of $\mathrm{ker}[A\,|\,\bm{1}]^\mathrm{t}$, so that $W\,[A\,|\,\bm{1}] = 0$.

\end{itemize}

\subsubsection*{The positive part of the kernel of the stoichiometric matrix} In a dynamically nontrivial $(n,m,r)$ CRN, $\mathrm{ker}_+\,\Gamma:=\mathrm{ker}\,\Gamma \cap \mathbb{R}^m_{+}$ is a proper cone of dimension $m-r$ in $\mathbb{R}^m$. It is useful to parameterise $\mathrm{ker}_+\,\Gamma$ as follows. We first parameterise an arbitrary cross-section, say $\mathcal{C}$, of $\mathrm{ker}_+\,\Gamma$, via a linear injective map $h \colon Y \subseteq \mathbb{R}^{m-r-1} \to \mathbb{R}^m_{+}$ with image $\mathcal{C}$. We then parameterise $\mathrm{ker}_+\,\Gamma$ as
\[
\mathrm{ker}_+\,\Gamma = \{\lambda h(\alpha)\,|\,\lambda \in \mathbb{R}_+, \alpha \in Y\}\,.
\]
Note that this parameterisation is valid in the special case where $m-r=1$, provided we now interpret $h(\alpha)$ as a constant vector, say $v \in \mathbb{R}^m_+$, in which case we have
\[
\mathrm{ker}_+\,\Gamma = \{\lambda v\,|\,\lambda \in \mathbb{R}_+\}\,.
\]

 \subsubsection*{The solvability condition for positive equilibria} Consider a dynamically nontrivial $(n,m,r)$ mass-action network and fix $\kappa \in \mathbb{R}^m_+$. Then $x \in \mathbb{R}^n_+$ is an equilibrium of the network with these rate constants if and only if $\kappa \circ x^A \in \mathrm{ker}_+\,\Gamma$, namely, after taking logarithms and rearranging, if there exist $\alpha \in Y$ and $\lambda \in \mathbb{R}_+$ such that
\begin{equation}
\label{eq:main}
\ln\,\kappa + [A\,|\,\bm{1}]\left(\begin{array}{r}\ln x\\-\ln \lambda\end{array}\right) = \ln h(\alpha)\,.
\end{equation}
By the Fredholm alternative, equation \eqref{eq:main} admits solutions if and only if the following solvability condition holds:
\begin{equation}
\label{eq:solvability}
W\,\ln\kappa = W\ln h(\alpha) \quad \mbox{i.e.,} \quad \kappa^W = h(\alpha)^W\,.
\end{equation}
Note that in the case that $r=m-1$, $h(\alpha)$ is a constant vector; and if $[A\,|\,\bm{1}]$ has rank $m$ so that $W$ is empty, we take the solvability condition to be trivially satisfied.

\subsubsection*{The Jacobian matrix of a mass-action network} The Jacobian matrix of a mass-action network at any point in the positive orthant takes the form
\begin{equation}
\label{eq:MAJac0}
J = \Gamma D_1 A D_2
\end{equation}
where $D_1$ and $D_2$ are positive diagonal matrices (depending on rate constants and state vector) and, as usual, $A = \Gamma_\ell^\mathrm{t}$ \cite{banaji:donnell:baigent:2007}. The Jacobian matrix of a dynamically nontrivial $(n,m,n)$ mass-action network, {\em evaluated at a positive equilibrium}, takes the specific form
\begin{equation}
\label{eq:MAJac}
J= \lambda \Gamma D_{h(\alpha)}AD_{1/x}\,,
\end{equation}
where $D_{h(\alpha)}$ is the positive diagonal matrix with $(D_{h(\alpha)})_{i,i} = h_i(\alpha)$, and $D_{1/x}$ is the positive diagonal matrix with $(D_{1/x})_{i,i} = 1/x_i$ \cite[Section 2.3]{banaji:boros:2023}.

We refer to an equilibrium of a mass-action network as {\em nondegenerate} if the Jacobian matrix, evaluated at the equilibrium, acts as a nonsingular transformation on the stoichiometric subspace. For characterisations of this property in terms of the so-called {\em reduced Jacobian determinant}, see \cite[Section 2.2]{banaji:pantea:2016}. Clearly, for an $(n,m,n)$ network, an equilibrium is nondegenerate if and only if the Jacobian matrix evaluated at the equilibrium is nonsingular. 

\begin{lemma}
\label{lem:degen} 
An $(n,m,n)$ CRN such that $\mathrm{rank}\,[A\,|\,\bm{1}] \leq n$ admits no positive nondegenerate equilibria. 
\end{lemma}

\begin{proof} Recall that the Jacobian matrix at a positive equilibrium has the form in \eqref{eq:MAJac}, namely, $J = \lambda \Gamma D_{h(\alpha)} A D_{1/x}$. Note that $\mathrm{rank}\,[A\,|\,\bm{1}] \leq n$ implies that either $\mathrm{rank}\,A <n$, or $\bm{1} \in \mathrm{im}\,A$ and we will show that in both cases $J$ is identically singular. If $\mathrm{rank}\,A <n$, then $J$ is clearly singular for all $\alpha \in Y$ and $x \in \mathbb{R}^n_+$ as it includes a factor of rank less than $n$. If $\mathrm{rank}\,A =n$ and $\bm{1} \in \mathrm{im}\,A$, say $\bm{1} = Az$ for some $z \in \mathbb{R}^n$, then again $J$ is singular as, for any $\alpha \in Y, x \in \mathbb{R}^n_+$, 
\[
\lambda \Gamma D_{h(\alpha)} A D_{1/x} (D_{1/x}^{-1}z) = \lambda \Gamma D_{h(\alpha)}\bm{1} = \lambda \Gamma h(\alpha) = 0\,,
\]
i.e., the nonzero vector $D_{1/x}^{-1}z$ lies in $\mathrm{ker}\,J$. 
\end{proof}

\begin{remark}
It is possible to show more generally that an $(n,m,r)$ mass-action network with $\mathrm{rank}\,[A\,|\,\bm{1}] \leq r$ admits positive equilibria for only an exceptional set of rate constants, and that these equilibria are always degenerate. The condition $\mathrm{rank}\,[A\,|\,\bm{1}] \leq r$ is sufficient, but not necessary, for the degeneracy of all equilibria.
\end{remark}

\subsection{Natural coordinates for $(n,m,n)$ mass-action networks admitting positive nondegenerate equilibria}
\label{sec:coords}
In this section, we consider an arbitrary $(n,m,n)$ mass-action network satisfying the two conditions (i) $m-n \geq 1$, and (ii) $\mathrm{rank}\,[A\,|\,\bm{1}] = n+1$. Note that the first condition holds if the network is dynamically nontrivial, a necessary condition for the existence of positive equilibria; and the second condition holds, by \Cref{lem:degen}, if the network admits a positive nondegenerate equilibrium. 

Any mass-action network satisfying conditions (i) and (ii) admits natural and useful parameter-dependent recoodinatisations, generalising the recoordinatisation presented for $(n,n+1,n)$ networks in \cite[Section 3.2]{banaji:boros:2023}. An important consequence (\Cref{thm:bifcodim} below) is that $(n,m,n)$ mass-action networks give rise, effectively, to $(m-2)$-parameter families of ODEs and hence can admit no bifurcations of codimension greater than $m-2$ unfolded by their rate constants. Another consequence (\Cref{thm:cusp} below) of particular interest in the context of this paper is that $(n,n+2,n)$ mass-action networks cannot admit cusp bifurcations. 

We assume that $m>n+1$; the easier case where $m=n+1$ is essentially covered in \cite[Section 3.2]{banaji:boros:2023}, or follows by setting the matrices $U$ and $W$ in what follows to be empty. Let $U$ be any $m \times (m-n-1)$ matrix such that $[A\,|\,\bm{1}\,|\,U]$ is nonsingular (such $U$ exists as $[A\,|\,\bm{1}]$ has rank $n+1$ by assumption). Let $\doublehat{G}:= [A\,|\,\bm{1}\,|\,U]^{-1}$, and write
\[
\doublehat{G} = \left(\begin{array}{c}G\\v\\W\end{array}\right) 
\]
where $G$ refers to the top $n \times m$ block of $\doublehat{G}$, and $v$ is its $(n+1)$th row. The columns of $W^\mathrm{t}$ clearly form a basis of $\mathrm{ker}\,[A\,|\,\bm{1}]^{\mathrm{t}}$.
As
\[
\left(\begin{array}{c}G\\v\\W\end{array}\right)\,(A\,|\,\bm{1}\,|\,U) = (A\,|\,\bm{1}\,|\,U)\,\left(\begin{array}{c}G\\v\\W\end{array}\right) = I_{m \times m}\,,
\]
it follows that
\begin{equation}
\label{eq:identity}
I_{m \times m} - AG = \bm{1}\,v + U\,W\,. 
\end{equation}
Now define the new coordinates $y$ on $\mathbb{R}^n_+$ by $y = \kappa^G\,\circ x$. We obtain the ODE for $y$ as follows:
\begin{eqnarray*}
  \dot y = \kappa^G \circ \dot x
         & = & \kappa^G \circ \Gamma\,(\kappa \circ (\kappa^{-G}\,\circ y)^A)\\
         & = & \kappa^G \circ \Gamma\,(\kappa^{I_{m \times m}-AG} \circ y^A)\\
         & = & \kappa^G \circ \Gamma\,(\kappa^{\bm{1}\,v + UW} \circ y^A) \quad \mbox{(from \eqref{eq:identity})}\\
  & = & \kappa^G \circ \Gamma\,((\kappa^v)^{\bm{1}} \circ (\kappa^W)^U \circ y^A)\\
         & = & \kappa^G \circ (\kappa^v)\bm{1} \circ \Gamma\,((\kappa^W)^U \circ y^A) \quad \mbox{(since $\kappa^{\bm{1}\,v} = (\kappa^v)^{\bm{1}} = \kappa^v\,\bm{1}$)}\\
         & = & \kappa^{G+\bm{1}v} \circ \Gamma\,((\kappa^W)^{U} \circ y^A)\,.
\end{eqnarray*}
Note that $\kappa^W$ is a vector of $m-n-1$ new positive parameters which we will refer to as the {\em inner parameters}, while $\kappa^{G+\bm{1}v}$ is a vector of $n$ new positive parameters which we will refer to as the {\em outer parameters}. We may subsequently rescale time to reduce the number of outer parameters by $1$ leaving an $(m-2)$-parameter family of ODEs. Moreover, only the $m-n-1$ inner parameters directly affect the equilibrium set; the remaining parameters act via a positive diagonal transformation on the whole vector field.

\begin{remark}
The matrices $U$ and $W$ in the construction above are not uniquely determined, and consequently, there is some freedom in the choice of inner parameters. It is possible that appropriate choices may simplify computations for certain networks. For example, planar S-systems (which arise from $(2,4,2)$ mass-action networks with a special stoichiometric matrix) can be written without inner parameters, see \cite[Eq.\ (2)]{boros:hofbauer:mueller:regensburger:2019}.
\end{remark}

\begin{remark}
In $y$ coordinates, positive equilibria occur when
\[
(\kappa^W)^{U} \circ y^A \in \ker_+\,\Gamma\,, \quad \mbox{namely,} \quad UW\,\ln\kappa + A\ln y = \bm{1}\ln \lambda  + \ln h(\alpha).
\]
Multiplying through by $W$ and noting that $WU$ is the identity, $WA = 0$, and $W\bm{1} = 0$ we obtain, as expected, the solvability condition \eqref{eq:solvability}, namely $\kappa^W = h(\alpha)^W$.
\end{remark}

\begin{example}
Consider the mass-action network
\[
2\X \overset{\kappa_1}{\longrightarrow} 3\X \qquad \X+\Y \overset{\kappa_2}{\longrightarrow} 2\Y \qquad \Y \overset{\kappa_3}{\longrightarrow} \0 \qquad \0 \overset{\kappa_4}{\longrightarrow} \Y
\]
Associated with the network are the following stoichiometric matrix and reactant matrix:
\[
\Gamma = \left(\begin{array}{rrrr}1&-1&0&0\\0&1&-1&1\end{array}\right),\quad \Gamma_\ell = \left(\begin{array}{rrrr}2&1&0&0\\0&1&1&0\end{array}\right)\,.
\]
The network gives rise to the ODE system $\dot x = \Gamma(\kappa \circ x^A)$ where, as usual, $A:= \Gamma_\ell^\mathrm{t}$. Following the construction above, we can define
\[
\overline{G} = \left(\begin{array}{rrrr}1&-1&1&-1\\0&0&1&-1\\1&-2&2&0\\-1&2&-2&1\end{array}\right), \quad \overline{G}^{-1} = \left(\begin{array}{rrrr}2&0&1&2\\1&1&1&2\\0&1&1&1\\0&0&1&1\end{array}\right).
\]
With these choices we obtain the ODE system
\begin{equation}
\label{eq:recoord_example}
\left(\begin{array}{c}\dot X\\\dot Y\end{array}\right)  = \left(\begin{array}{c}\alpha_1\\\alpha_2\end{array}\right) \circ \left(\begin{array}{rrrr}1&-1&0&0\\0&1&-1&1\end{array}\right)\left(\begin{array}{c}\alpha_3^2X^2\\\alpha_3^2XY\\\alpha_3 Y\\\alpha_3\end{array}\right) = \left(\begin{array}{c}\alpha_1(\alpha_3^2X^2 - \alpha_3^2XY)\\\alpha_2(\alpha_3^2XY - \alpha_3Y + \alpha_3)\end{array}\right)\,,
\end{equation}
where 
\[
X = \kappa_1\kappa_2^{-1}\kappa_3\kappa_4^{-1}x, \quad Y = \kappa_3\kappa_4^{-1}y, \quad \alpha_1 = \kappa_1^2\kappa_2^{-3}\kappa_3^3\kappa_4^{-1}, \quad \alpha_2 = \kappa_1\kappa_2^{-2}\kappa_3^3\kappa_4^{-1}, \quad \alpha_3 = \kappa_1^{-1}\kappa_2^{2}\kappa_3^{-2}\kappa_4\,.
\]
Here $\alpha_1$ and $\alpha_2$ are the outer parameters, while $\alpha_3$ is the inner parameter. After some grouping of parameters and time-rescaling, \eqref{eq:recoord_example} simplifies to
\[
\begin{array}{rcl}\dot X &=& X^2 - XY\,,\\\dot Y & = & \beta(\gamma XY - Y + 1)\,,\end{array} 
\]
namely, we have only a two-parameter family of ODEs, with only one parameter affecting the equilibrium set.
\end{example}

\subsection{The number of positive nondegenerate equilibria in $(n,m,n)$ mass-action networks}

\begin{lemma}
\label{lem:biject}
Consider a dynamically nontrivial $(n,m,n)$ mass-action network such that $\mathrm{rank}\,[A\,|\,\bm{1}] = n+1$. 
For any fixed $\kappa$, there is a smooth bijection between the set of positive equilibria and solutions $\alpha \in Y$ to the solvability equation \eqref{eq:solvability}, namely $\kappa^W = h(\alpha)^W$. 
\end{lemma}

\begin{proof}
We present the proof in the case $m \geq n+2$; the case $m=n+1$ is easy. 
Recall that given any fixed $\kappa \in \mathbb{R}^m_+$, $x \in \mathbb{R}^n_+$ is a positive equilibrium of the CRN if and only if there exist $\lambda \in \mathbb{R}_+$ and $\alpha \in Y$ such that $(x,\lambda,\alpha)$ satisfies equation \eqref{eq:main}, which occurs if and only if $\kappa^W = h(\alpha)^W$. Note that if, for fixed $\kappa$ and $\alpha$, there exists $(x, \lambda)$ solving \eqref{eq:main} then, by injectivity of the linear transformation corresponding to $[A\,|\,\bm{1}]$, these are unique. Moreover, given $(\kappa, x_1, \lambda_1, \alpha_1)$ and $(\kappa, x_2, \lambda_2, \alpha_2)$ solving \eqref{eq:main} with $\alpha_1 \neq \alpha_2$, we must have $x_1 \neq x_2$, for otherwise a quick computation reveals that $h(\alpha_1) = (\lambda_2/\lambda_1) h(\alpha_2)$ which is impossible as $h(\alpha_1)$ and $h(\alpha_2)$ are distinct points on a {\em cross-section} of $\mathrm{ker}_+\,\Gamma$. These observations confirm that solutions $\alpha \in Y$ to $\kappa^W = h(\alpha)^W$ are, in fact, in one-to-one correspondence with equilibria. In order to see that the correspondence is smooth, we can explicitly write 
\[
x = (h(\alpha)/\kappa)^G\,,
\]
where $G$ is the top $n \times m$ block of any left inverse of $[A\,|\,\bm{1}]$.
\end{proof} 

The previous lemma allows us to see that an $(n,m,n)$ network with an insufficient number of distinct reactant complexes, forbids multiple positive nondegenerate equilibria.
\begin{theorem}
\label{thm:nplus2sources}
An $(n,m,n)$ mass-action network with fewer than $n+2$ distinct reactant complexes admits no more than one positive nondegenerate equilibrium. 
\end{theorem}
\begin{proof}
Consider an $(n,m,n)$ mass-action network. If it has fewer than $n+1$ distinct reactant complexes, then $\mathrm{rank}\,[A\,|\,\bm{1}]\leq n$ (it has no more than $n$ distinct rows); and equilibria, whenever they exist, are always degenerate by \Cref{lem:degen}. In this case the claim is trivially true. 

So now suppose that the network has exactly $n+1$ distinct reactant complexes and $[A\,|\,\bm{1}]$ has rank $n+1$. If $m=n+1$, the claim follows by \cite[Lemma 3.1]{banaji:boros:2023}. So we assume that $m > n+1$. It is easily seen that $\mathrm{ker}\,[A\,|\,\bm{1}]^\mathrm{t}$ has a basis consisting of $m-n-1$ vectors each containing a single entry $1$, a single entry $-1$, and all other entries $0$ (each such vector corresponds to a pair of identical rows in $[A\,|\,\bm{1}]$). Let $W$ be a matrix whose rows consist of such a set of basis vectors. From \Cref{lem:biject}, equilibria are in smooth one-to-one correspondence with solutions $\alpha$ to the equation $\kappa^W = h(\alpha)^W$. Recalling that $h(\alpha)$ is linear, and clearing denominators, for each fixed $\kappa$ this is a system of $m-n-1$ {\em linear} equations in $m-n-1$ unknowns $\alpha$ which thus admits either no solutions, a unique solution, or a continuum of solutions. Consequently, for each fixed $\kappa$, the CRN either admits no equilibria, one equilibrium or a continuum of equilibria. Since a continuum of equilibria necessarily consists of degenerate equilibria, the claim is proved.
\end{proof}


%

\subsection{Lemmas relevant to Andronov--Hopf bifurcation in mass-action networks}

The first lemma in this section tells us that a ``mixed'' reactant complex is necessary for the existence of a periodic orbit or a positive equilibrium with a pair of purely imaginary eigenvalues.

\begin{lemma} \label{lem:nonnegative_off_diagonal}
For a planar mass-action network without a reactant complex which includes both species of the network, the off-diagonal entries of the Jacobian matrix on the positive quadrant are nonnegative. Hence, there is no periodic orbit, and the Jacobian matrix on the positive quadrant can never have a pair of purely imaginary eigenvalues.
\end{lemma}
\begin{proof}
Recall \eqref{eq:MAJac0}, i.e., the Jacobian matrix of any mass-action network on the positive orthant takes the form $J = \Gamma D A E$ where $A = \Gamma_\ell^\mathrm{t}$ and $D$ and $E$ are positive diagonal matrices. Suppose that the reactant complexes of the network never involve both chemical species, i.e., $A_{j,1}>0 \Rightarrow A_{j,2}=0$ and $A_{j,2}>0 
\Rightarrow A_{j,1}=0$. Consequently, since $A_{j,i}=0$ implies $\Gamma_{i,j}\geq 0$, 
we have $\Gamma_{1,j}A_{j,2} \geq 0$ and $\Gamma_{2,j}A_{j,1} \geq 0$ for all $j = 1, \ldots, m$. Further,  
\[
J_{1,2} = \sum_{j=1}^m \Gamma_{1,j}\,D_{j,j}\,A_{j,2}\,E_{2,2}\, \text{  and  } J_{2,1} = \sum_{j=1}^m \Gamma_{2,j}\,D_{j,j}\,A_{j,1}\,E_{1,1}\,
\]
and hence, $J_{1,2} \geq 0$ and $J_{2,1} \geq 0$.

Therefore, the system is cooperative and thus cannot have periodic orbits \cite[Theorem 3.2.2]{smith:1995}. Moreover, $\tr J = 0$ and $\det J > 0$ can never occur simultaneously, and hence, a pair of nonzero imaginary eigenvalues is forbidden.
\end{proof}

\begin{remark}
We note that \Cref{lem:nonnegative_off_diagonal} can be extended (with an analogous proof) to any number of species: if each reactant complex is of the form $k\X_i$ for some $k\geq0$ then the off-diagonal entries of the Jacobian matrix on the positive quadrant are nonnegative. Hence, the system is cooperative, and it can have no attracting limit cycle, see 
\cite[Theorem 1.2.2 and Proposition 4.3.4]{smith:1995}.
\end{remark}

The following result is an immediate consequence of \Cref{lem:nonnegative_off_diagonal} for quadratic networks.

\begin{lemma}\label{lem:mixed_source}
A planar, quadratic mass-action network admitting a periodic orbit must include the reactant complex $\X+\Y$.
\end{lemma}

By the next lemma, the presence of the reaction $2\X \to 3\X$ or $2\Y \to 3\Y$ is necessary for an Andronov--Hopf bifurcation to occur in the mass-action networks of our interest.

\begin{lemma}\label{lem:2Xto3X}
A planar, quadratic, trimolecular mass-action network admitting an isolated periodic orbit or a positive equilibrium with a pair of complex conjugate eigenvalues that cross the imaginary axis must include one of the reactions $2\X \to 3\X$ and $2\Y \to 3\Y$.
\end{lemma}
\begin{proof}
As in the proof of \cite[Lemma 3]{banaji:boros:hofbauer:2024a}, after multiplication by the Dulac function $\frac{1}{xy}$, it is clear that the only way to have positive divergence is if at least one of the reactions $2\X \to 3\X$ and $2\Y \to 3\Y$ is present.
\end{proof}

We conclude this section by recalling \cite[Lemma 2]{banaji:boros:hofbauer:2024a}.

\begin{lemma}\label{lem:nitecki}
A planar mass-action network whose reactant complexes lie on a line admits no periodic orbit.
\end{lemma}

\section{Fold bifurcation}
\label{sec:fold}

In this section, we discuss \emph{fold} (sometimes termed \emph{saddle-node}) bifurcations of positive equilibria in small networks. Though this codimension-one bifurcation happens on a one-dimensional center manifold, reasonable assumptions on the molecularity limit its occurrence in rank-one networks. In \Cref{subsec:fold_rank-one}, we characterise quadratic, rank-one networks that admit a fold bifurcation. In particular, we find that bimolecular, rank-one networks do not admit a fold bifurcation. In \Cref{subsec:fold_planar}, we start by identifying all quadratic, trimolecular $(2,4,2)$ networks that admit a fold bifurcation. We list all networks in this class that are bimolecular. In fact, these networks are the \emph{smallest} planar, bimolecular networks that admit a fold bifurcation, meaning that any planar, bimolecular mass-action network admitting fold bifurcation must have at least $4$ reactions, and rank at least two. Finally, we highlight those bimolecular $(2,4,2)$ networks that admit multiple nonnegative asymptotically stable equilibria, and remark on an erroneous claim about such networks in previously published work.



\subsection{Rank-one networks}
\label{subsec:fold_rank-one}

It is straightforward to verify that the quadratic, trimolecular $(1,3,1)$ network $\0 \rightleftarrows \X$, $2\X \to 3\X$ admits two positive nondegenerate equilibria that are born via a fold bifurcation. In fact, remarkably, this network, or another one that is simply equivalent to it, appears as an induced subnetwork in {\em every} quadratic, rank-one mass-action network that admits multiple positive nondegenerate equilibria. We thus have the following entirely {\em combinatorial} characterisation of quadratic, rank-one mass-action networks admitting multiple positive nondegenerate equilibria.

\begin{theorem}\label{thm:quadratic_rank-one_fold}
For a quadratic, rank-one mass-action network the following are equivalent.
\begin{enumerate}[(i)]
\item It admits multiple positive nondegenerate equilibria.
\item It admits a nondegenerate fold bifurcation of a positive equilibrium.
\item It includes one of the (simply equivalent) networks $\0 \to a\X$, $\X \to \0$, $2\X \to b\X$ ($a\geq1$, $b\geq3$) as an induced subnetwork.
\end{enumerate}
\end{theorem}
\begin{proof}
The implication (iii) $\Longrightarrow$ (ii) holds by the results on the inheritance of bifurcations, see \Cref{subsec:inheritance}, and \Cref{rem:induced} in particular. The implication (ii) $\Longrightarrow$ (i) is trivial. Finally, (i) $\Longrightarrow$ (iii) follows from \Cref{lem:quadratic_rank-one_fold} below.
\end{proof}

The following corollary is an immediate consequence of \Cref{thm:quadratic_rank-one_fold} and also follows from \cite[Theorem 4.1]{pantea:voitiuk:2022} by Pantea and Voitiuk.

\begin{corollary}
A bimolecular, rank-one mass-action network admits at most one positive nondegenerate equilibrium, and hence, does not admit a nondegenerate fold bifurcation.
\end{corollary}


The previous two results depend on the following lemma. 
\begin{lemma}\label{lem:quadratic_rank-one_fold}
A quadratic, rank-one mass-action network admits multiple positive nondegenerate equilibria if and only if it has at most two nontrivial species and, after removal of trivial species if any, it includes three reactions that form a network which is simply equivalent to one of (1), (2a), (2b) or (2c) below.
\begin{center}
\begin{tabular}{l*{3}{r@{\hspace{4pt}$\to$\hspace{4pt}}l}}
\arrayrulecolor{mygray}
(1)  &  $\0$ & $\X$     & $\X$    & $\0$ & $2\X$ & $3\X$ \\
\hline
\multicolumn{7}{c}{} \\[-10pt]
(2a) &  $\0$ & $\X+ \Y$ & $\X+\Y$ & $\0$ & $2\X$ & $3\X+\Y$ \\
\hline
\multicolumn{7}{c}{} \\[-10pt]
(2b) &  $\Y$ & $\X+2\Y$ & $\X+\Y$ & $\0$ & $2\X$ & $3\X+\Y$ \\
\hline
\multicolumn{7}{c}{} \\[-10pt]
(2c) & $2\Y$ & $\X+3\Y$ & $\X+\Y$ & $\0$ & $2\X$ & $3\X+\Y$
\end{tabular}
\end{center}
\end{lemma}
\begin{proof}
It is straightforward to verify that the network $\0 \to \X$, $\X \to \0$, $2\X \to 3\X$ admits multiple positive nondegenerate equilibria. Networks (2a), (2b), (2c) are obtained from (1) by adding a dependent species (enlargement E3), and hence, by \Cref{lem:inherit}, they also admit multiple positive nondegenerate equilibria. By inheritance, any rank-one network that is obtained by adding reactions and/or species to a network that is simply equivalent to one of (1), (2a), (2b), (2c) also admits multiple positive nondegenerate equilibria (see \Cref{rem:induced}).

Observe that the removal of trivial species does not affect the capacity of a network for multiple positive nondegenerate equilibria. Now fix a quadratic, rank-one mass-action network with no trivial species, and fix the rate constants and a positive stoichiometric class $\mathcal{P}$ with multiple positive nondegenerate equilibria. Let $v \in \mathbb{R}^n$ denote the generating vector of the stoichiometric subspace. Since the network has no trivial species, each entry of $v$ is nonzero.
\begin{description}
\item[Case 1] If the entries of $v$ have mixed signs then $\mathcal{P}$ is bounded. Eliminate all but one variable using the $n-1$ linear conservation relations and thereby obtain the quadratic, scalar differential equation $\dot{x}=f(x)$ on a compact interval $[0,d]$ (the projection of $\overline{\mathcal{P}}$ to the $x$-axis). Since $[0,d]$ is forward invariant, $f(0)\geq0$ and $f(d)\leq0$ hold. This contradicts the existence of two distinct roots of $f$ in the open interval $(0,d)$ (recall that $f$ is quadratic). 
\item[Case 2] If the entries of $v$ are all positive (or all negative) then the dynamical nontriviality of the network implies that there is a reaction in which each species is consumed. Hence, there are at most two species (recall that the network is quadratic). In the single-species case, it is straightforward to see that the existence of multiple positive nondegenerate equilibria is possible only if the network includes the reactions $\0 \to a\X$, $\X \to 0$, $2\X \to b\X$ for some  $a\geq1$ and $b\geq3$. In the two-species case, the network contains the reaction $\X+\Y\to\0$ (the only reaction with bimolecular reactant complex in which both species are consumed). A short calculation shows that if the reactant complexes of all the other reactions are at most linear then multiple positive equilibria are forbidden. Hence, $2\X\to(2+b)\X+b\Y$ or $2\Y\to b\X+(2+b)\Y$ for some $b\geq1$ is part of the network. If both, we have found a network that is simply equivalent to (2c). If only one of them, say $2\X\to(2+b)\X+b\Y$, then a reaction with reactant complex not involving the species $\X$ is present in the network (otherwise the line $x=0$ consists of equilibria, and hence, no multiple positive equilibria are possible on the positive part of the line $x=y+d$). Therefore, $\0 \to b\X+b\Y$ or $\Y \to b\X+(1+b)\Y$ for some $b\geq1$ is part of the network, implying that a network simply equivalent to (2a) or (2b) is included.
\end{description}
\end{proof}


\subsection{Planar networks}
\label{subsec:fold_planar}

In \Cref{subsec:fold_rank-one} we have characterised nondegenerate fold bifurcations in quadratic, rank-one mass-action networks. In this section we study fold bifurcations in quadratic, trimolecular, planar, rank-two mass-action networks, with special attention on the bimolecular networks. By \Cref{thm:nplus2sources}, a planar, rank-two network can admit multiple positive nondegenerate equilibria only if it has at least four distinct reactant complexes, and hence, at least four reactions. Therefore, our focus will be on four-reaction networks. However, we remark that when allowing more than two species, a nondegenerate fold bifurcation occurs already in bimolecular, three-reaction networks, see \cite[Section 5.1]{banaji:boros:hofbauer:2024a}. The following lemma and theorem summarise our main findings regarding fold bifurcations in quadratic, trimolecular $(2,4,2)$ mass-action networks. The enumerations of the sets of networks appearing in the results to follow are performed in the Mathematica Notebook \cite{boros_github}.

\begin{lemma}\label{lem:quadratic_trimolecular_fold_base_set}
There are 5897 dynamically nonequivalent, dynamically nontrivial, quadratic, trimolecular $(2,4,2)$ mass-action networks with four distinct reactant complexes. Out of these 5897 networks, 5864 admit a positive nondegenerate equilibrium.
\end{lemma}

\begin{theorem}\label{thm:834_quadratic_trimolecular_fold}
Out of the 5864 mass-action networks in \Cref{lem:quadratic_trimolecular_fold_base_set}, 834 admit a positive equilibrium with a zero eigenvalue. In 3 networks, the zero eigenvalue always has an algebraic multiplicity of two. Each of the remaining 831 networks admits a nondegenerate fold bifurcation. Out of these 831 networks, in 825 cases the zero eigenvalue can be accompanied by a negative eigenvalue, and in 39 cases the zero eigenvalue can be accompanied by a positive eigenvalue, with the intersection of these two sets consisting of 33 networks. 
\end{theorem}

\begin{remark} \label{rem:fold_quadratic}
We have collected several comments on the networks appearing in \Cref{lem:quadratic_trimolecular_fold_base_set} and \Cref{thm:834_quadratic_trimolecular_fold}. 
\begin{enumerate}[(a)]
\item The 831 networks in \Cref{thm:834_quadratic_trimolecular_fold} that admit a fold bifurcation fall into 639 diagonally nonequivalent classes \cite{boros_github}.

\item \label{item:cusp} None of the 831 networks in \Cref{thm:834_quadratic_trimolecular_fold} that admit a fold bifurcation can admit a cusp bifurcation, for the reason that cusp bifurcation is forbidden in $(2,4,2)$ networks (see \Cref{thm:cusp} in \Cref{sec:other}).

\item The 33 networks in \Cref{thm:834_quadratic_trimolecular_fold} for which the second eigenvalue can change sign are listed in \Cref{tab:B-T}. In \Cref{sec:B-T}, we will see that most of these networks admit a nondegenerate Bogdanov--Takens bifurcation, with two of them admitting a particular degenerate form of this bifurcation which we will refer to as a {\em vertical} Bogdanov--Takens bifurcation.

\item\label{item:fold_quadratic_2d_center_manifold} The 3 networks in \Cref{thm:834_quadratic_trimolecular_fold} for which the zero eigenvalue cannot have an algebraic multiplicity of one are
\begin{center}
\begin{tabular}{*{4}{r@{\hspace{4pt}$\to$\hspace{4pt}}l}}
\arrayrulecolor{mygray}
$2\X$ & $3\X$ & $\X+\Y$ & $2\Y$ & $\Y$ & $\0$ & $2\Y$ & $\X+2\Y$ \\
\hline
\multicolumn{8}{c}{}\\[-10pt]
$2\X$ & $3\X$ & $\X+\Y$ & $3\Y$ & $\Y$ & $\0$ & $2\Y$ & $\X+2\Y$ \\
\hline
\multicolumn{8}{c}{}\\[-10pt]
$2\X$ & $3\X$ & $\X+\Y$ &  $\0$ & $\Y$ & $2\Y$ & $2\Y$ & $\X+2\Y$
\end{tabular}
\end{center}
For these 3 networks (of which the first and the second are diagonally equivalent), similarly to a fold bifurcation, two positive equilibria are born; however, at the bifurcation point, the Jacobian matrix is nilpotent of index two. We may regard the resulting bifurcation as an incomplete unfolding of a double zero singularity. Notice that the first and the second networks are obtained from the \emph{generalised LVA} network \cite{farkas:noszticzius:1985,simon:1992,banaji:boros:hofbauer:2024a}
\begin{align} \label{eq:LVA}
2\X \to 3\X \qquad \X+\Y \to b\Y \qquad \Y \to \0 \quad \text{with $b=2$ or $3$}
\end{align}
via enlargement E1.

\item\label{item:fold_quadratic_E6} Out of the 831 networks in \Cref{thm:834_quadratic_trimolecular_fold}, in 15 cases the fold bifurcation is inherited from the special one-species networks $\0 \to a\X$, $\X \to \0$, $2\X \to 3\X$ (with $a=1,2,3$) appearing in \Cref{thm:quadratic_rank-one_fold} by applying enlargement E6. These 15 networks are:  
\begin{center}
\begin{tabular}{lll}
\arrayrulecolor{mygray}
$\0 \to C \to \X$ & $\X \to \0$ & $2\X \to 3\X$ \\
\hline
\multicolumn{3}{c}{} \\[-10pt]
$\0 \to C \to 2\X$ & $\X \to \0$ & $2\X \to 3\X$ \\
\hline
\multicolumn{3}{c}{} \\[-10pt]
$\0 \to C \to 3\X$ & $\X \to \0$ & $2\X \to 3\X$ \\
\hline
\multicolumn{3}{c}{} \\[-10pt]
$\0 \to \X$ & $\X \to C \to \0$ & $2\X \to 3\X$ \\
\hline
\multicolumn{3}{c}{} \\[-10pt]
$\0 \to \X$ & $\X \to \0$ & $2\X \to C \to 3\X$ \\
\end{tabular}
\end{center}
where the inserted complex $C$ is any of $\Y$, $2\Y$, or $\X+\Y$. The remaining 816 networks are atoms of fold bifurcation (see \Cref{subsec:inheritance}).

\item\label{item:fold_quadratic_vertical} In all the $5897-5864=33$ networks in \Cref{lem:quadratic_trimolecular_fold_base_set} which do not admit a positive, nondegenerate equilibrium, a \emph{vertical fold} bifurcation occurs: at the critical parameter values, there is a curve of equilibria. For each of the 33 networks, this curve is a straight line (either through the origin or vertical or horizontal), see \cite{boros_github}.

\item\label{item:fold_quadratic_unique_equil} If we remove from the $5864$ networks in \Cref{lem:quadratic_trimolecular_fold_base_set} the $831$ networks in \Cref{thm:834_quadratic_trimolecular_fold} which admit a nondegenerate fold bifurcation, and the $3$ networks with the bifurcation described in \eqref{item:fold_quadratic_2d_center_manifold} above, we can confirm computationally that the remaining $5030$ networks admit at most one positive equilibrium \cite{boros_github}. 

\item\label{item:fold_quadratic_at_most_two_equil} 
We can confirm computationally that for all the 834 networks in \Cref{thm:834_quadratic_trimolecular_fold}, for any pair of positive equilibria, the Jacobian determinant is negative at one, and positive at the other \cite{boros_github}. Hence, these networks have at most two positive equilibria.

\item As a consequence of \eqref{item:fold_quadratic_unique_equil} and \eqref{item:fold_quadratic_at_most_two_equil}, no quadratic, trimolecular $(2,4,2)$ network admits three positive nondegenerate equilibria. This is not automatic from \eqref{item:cusp} above, as a nondegenerate cusp bifurcation is sufficient but not necessary for the existence of three positive nondegenerate equilibria, as discussed further in \Cref{sec:other}.



\end{enumerate}
\end{remark}


In the rest of this section, we discuss bimolecular $(2,4,2)$ networks that admit a fold bifurcation. We start by specialising \Cref{lem:quadratic_trimolecular_fold_base_set} and \Cref{thm:834_quadratic_trimolecular_fold} to the bimolecular case, and then we make some comments.

\begin{lemma}\label{lem:bimolecular_fold_base_set}
There are 838 dynamically nonequivalent, dynamically nontrivial, bimolecular $(2,4,2)$ mass-action networks with four distinct reactant complexes. Out of these 838 networks, 829 admit a positive nondegenerate equilibrium.
\end{lemma}

\begin{theorem}\label{thm:30_bimolecular_fold}
Out of the 829 networks in \Cref{lem:bimolecular_fold_base_set}, 30 admit a positive equilibrium with a zero eigenvalue, see \Cref{tab:fold_bimolecular}. In each of the 30 cases, the other eigenvalue is negative, and the equilibrium undergoes a nondegenerate fold bifurcation.
\end{theorem}

\begin{table}[h!t]
\centering
\begin{tikzpicture}
\node[below] at (0,0) {\footnotesize
\begin{NiceTabular}{|w{c}{4mm}r@{\hspace{4pt}$\to$\hspace{4pt}}l| *{8}{w{c}{1mm}}|}
\hline
\multicolumn{3}{c}{\begin{tikzpicture}
    \node at (0,0.5) {Group 1};
    \node[rounded corners,fill=fillbg2,draw] at (0,0) {\begin{tabular}{r@{\hspace{4pt}$\to$\hspace{4pt}}l}
$\X$ & $2\X$ \\ $\X+\Y$ & $\0$
\end{tabular}};
\end{tikzpicture}} &
\RowStyle{\rotate} $2\X\to\0$ & $2\X\to\Y$ & $2\X\to2\Y$ & $\Y\to\X$ & $\Y\to2\X$ & $\Y\to\X+\Y$ & $2\Y\to\X$ & $2\Y\to2\X$ \\
\hline
& $\0$ &  $\Y$ & $\bullet$ & $\bullet$ & $\bullet$ & $\bullet$ & $\bullet$ & $\bullet$ & $\bullet$ & $\bullet$ \\
& $\Y$ & $2\Y$ & & & & & & & $\bullet$ & $\bullet$ \\
\hline
\end{NiceTabular}
};

\node[below] at (8,0) {\footnotesize
\begin{NiceTabular}{|w{c}{4mm}r@{\hspace{4pt}$\to$\hspace{4pt}}l| *{8}{w{c}{1mm}}|}
\hline
\multicolumn{3}{c}{\begin{tikzpicture}
    \node at (0,0.5) {Group 2};
    \node[rounded corners,fill=fillbg2,draw] at (0,0) {\begin{tabular}{r@{\hspace{4pt}$\to$\hspace{4pt}}l}
$\Y$ & $2\X$ \\ $\X+\Y$ & $2\Y$
\end{tabular}};
\end{tikzpicture}} &
\RowStyle{\rotate} $\0\to\Y$ & $\0\to\X+\Y$ & $\X\to\Y$ & $\X\to\X+\Y$ & $\X\to2\Y$ & $2\Y\to\0$ & $2\Y\to\X$ & $2\Y\to2\X$ \\
\hline
& $\X$  & $\0$ & $\bullet$ & $\bullet$ & & & & \orange{$\bullet$} & \orange{$\bullet$} & \\
& $2\X$ & $\0$ & $\bullet$ & $\bullet$ & $\bullet$ & $\bullet$ & $\bullet$ & \orange{$\bullet$} & \orange{$\bullet$} & \orange{$\bullet$} \\
\hline
\end{NiceTabular}
};

\node[below] at (4,-3) {\footnotesize
\begin{NiceTabular}{|w{c}{1mm}r@{\hspace{4pt}$\to$\hspace{4pt}}l| *{8}{w{c}{1mm}}|}
\hline
\multicolumn{3}{c}{\begin{tikzpicture}
    \node at (0,0.5) {Group 3};
    \node[rounded corners,fill=fillbg2,draw] at (0,0) {\begin{tabular}{r@{\hspace{4pt}$\to$\hspace{4pt}}l}
$\Y$ & $2\X$ \\ $2\X$ & $2\Y$
\end{tabular}};
\end{tikzpicture}} &
\RowStyle{\rotate} $\0\to\X$ & $\0\to\Y$ & $\0\to\X+\Y$ & $2\Y\to\0$ & $2\Y\to\X$ & $\X+\Y\to\0$ & $\X+\Y\to\X$ & $\X+\Y\to\Y$ \\
\hline
& $\X$ & $\0$ & $\bullet$ & $\bullet$ & $\bullet$ & \orange{$\bullet$} & \orange{$\bullet$} & \orange{$\bullet$} & \orange{$\bullet$} & \orange{$\bullet$} \\
\hline
\end{NiceTabular}
};
\end{tikzpicture}
\caption{The 30 dynamically nonequivalent, bimolecular $(2,4,2)$ mass-action networks that admit a fold bifurcation (see \Cref{thm:30_bimolecular_fold}). The networks are partitioned into 3 groups. In each group, the first and the second reactions are the same for all networks (indicated in the top left corner), while the third and the fourth reactions are the row- and the column headers, respectively. A bullet indicates a fold bifurcation, while networks without a bullet do not admit a fold bifurcation. For the 10 networks marked with an orange bullet, the origin is an asymptotically stable equilibrium, and hence, these networks admit multiple attracting equilibria in the nonnegative quadrant. The 20 networks with a black bullet admit at most one asymptotically stable nonnegative equilibrium.}
\label{tab:fold_bimolecular}
\end{table}

\begin{remark} \label{rem:fold_bimolecular}
We have collected several comments on the networks appearing in \Cref{lem:bimolecular_fold_base_set} and \Cref{thm:30_bimolecular_fold}. 

\begin{enumerate}[(a)]

\item The 30 networks in \Cref{tab:fold_bimolecular} are diagonally nonequivalent, although some are diagonally equivalent to other quadratic, trimolecular networks \cite{boros_github}.

\item From \Cref{rem:fold_quadratic} \eqref{item:fold_quadratic_E6} above, all 30 networks in \Cref{tab:fold_bimolecular} are atoms of fold bifurcation (see \Cref{subsec:inheritance}). In fact, this also follows from \Cref{thm:quadratic_rank-one_fold,thm:nplus2sources}, which tells us that these 30 networks are the \emph{smallest} planar, bimolecular networks that admit a fold bifurcation (meaning that no planar, bimolecular network with fewer than $4$ reactions or molecularity less than two admits a fold bifurcation).

\item\label{item:fold_bimolecular_trace_negative} For all the 30 networks in \Cref{tab:fold_bimolecular}, the trace of the Jacobian matrix is negative at any positive equilibrium. This follows directly from the proof of \cite[Theorem 4.1]{boros:hofbauer:2022}, where we show that the divergence at any positive equilibrium is negative, except for  some special systems that admit only a unique positive equilibrium.

\item \label{item:fold_bimolecular_at_most_two_equil} We can confirm that other than the 30 mass-action networks in \Cref{tab:fold_bimolecular}, no bimolecular $(2,4,2)$ networks admit multiple positive nondegenerate equilibria (see \Cref{rem:fold_quadratic} \eqref{item:fold_quadratic_unique_equil}). Further, amongst the 30 networks appearing in \Cref{thm:30_bimolecular_fold} none admit more than two positive nondegenerate equilibria (see \Cref{rem:fold_quadratic} \eqref{item:fold_quadratic_at_most_two_equil}). Furthermore, whenever there are two positive equilibria, one of them is a saddle, and the other one is linearly stable (this follows from \Cref{rem:fold_quadratic} \eqref{item:fold_quadratic_at_most_two_equil} and bimolecularity as in \eqref{item:fold_bimolecular_trace_negative} above).

\item In \cite{joshi:shiu:2013}, Joshi and Shiu studied bimolecular, fully open mass-action networks with two nonflow reactions that admit multiple positive nondegenerate equilibria. Three networks in \Cref{tab:fold_bimolecular} (one in Group 2 and two in Group 3) have only two nonflow reactions. Hence, three fully open networks in \cite[Fig. 3]{joshi:shiu:2013} inherit the fold bifurcation (and thereby the capacity for multiple positive nondegenerate equilibria) from a network with only two flow reactions instead of four; the inheritance is via enlargement E1 or E2 \cite{banaji:boros:hofbauer:2024b}.

\item In all the $838-829=9$ networks in \Cref{lem:bimolecular_fold_base_set} which do not admit a positive, nondegenerate equilibrium, a \emph{vertical fold} bifurcation occurs as these networks are among the 33 networks in \Cref{rem:fold_quadratic} \eqref{item:fold_quadratic_vertical}. In each case, at the critical value, there is a line of equilibria (either through the origin or vertical or horizontal), see \cite{boros_github}.

\end{enumerate}
\end{remark}

We close this section with a discussion of {\em multistability} in bimolecular $(2,4,2)$ mass-action networks. Although no bimolecular $(2,4,2)$ network admits multiple asymptotically stable positive equilibria (\Cref{rem:fold_bimolecular} \eqref{item:fold_bimolecular_at_most_two_equil}), there might coexist two asymptotically stable \emph{nonnegative} equilibria in these networks. Indeed, for some rate constants, both the origin and a positive equilibrium of the network
\begin{align}\label{eq:wilhelm}
\begin{tabular}{*{4}{r@{\hspace{4pt}$\to$\hspace{4pt}}l}}
$\Y$ & $2\X$ & $2\X$ & $\X+\Y$ & $\X+\Y$ & $\Y$ & $\X$ & $\0$ \\
\end{tabular}
\end{align}
discussed by Wilhelm \cite[Eq.\ (2)]{wilhelm:2009} (which is simply equivalent to the last network in Group 3 in \Cref{tab:fold_bimolecular}) are asymptotically stable.
In \cite{wilhelm:2009}, Wilhelm's goal was to find the \emph{smallest} bimolecular mass-action network that admits multiple stable equilibria in the nonnegative orthant. His criterion for being the smallest was to have 1) fewest species, 2) fewest reactions, and 3) fewest terms in the mass-action differential equation; in this order of importance. The mass-action differential equation associated with network \eqref{eq:wilhelm} has 6 terms, the origin is an asymptotically stable equilibrium for all rate constants, and one of the positive equilibria that is born via a fold bifurcation is also asymptotically stable. Wilhelm claims that network \eqref{eq:wilhelm} is the only bimolecular $(2,4,2)$ network with these properties. Our analysis shows that there are further such networks (suggesting that Wilhelm's proof is incomplete). We have found that out of the 30 networks in \Cref{tab:fold_bimolecular}, the origin is an equilibrium in 15 cases. Out of these 15 cases, the origin is asymptotically stable in 10 cases (the networks in \Cref{tab:fold_bimolecular} that are marked with an orange bullet). Out of the 10 networks, 5 (including Wilhelm's network) have 6 terms in their differential equation, and another 5 have 7 terms. Hence, following Wilhelm's criteria for being the \emph{smallest}, there are 5 such networks. We summarise some of these findings in the following lemma.

\begin{lemma}
\label{lem:bistable}
Out of the 30 networks in \Cref{thm:30_bimolecular_fold} (see \Cref{tab:fold_bimolecular}), 10 admit multiple asymptotically stable equilibria in the nonnegative quadrant (the networks marked with an orange bullet in \Cref{tab:fold_bimolecular}). In all 10 cases, one of the asymptotically stable equilibria is the origin. 
\end{lemma}
\begin{proof}
For the 7 networks with an orange bullet and the reaction $\X \to \0$, the origin is a linearly stable equilibrium. For the remaining 3 networks with an orange bullet (those with the reaction $2\X \to \0$), the Jacobian matrix at the origin is $\begin{pmatrix} 0 & 2\kappa_1 \\ 0 & -\kappa_1 \end{pmatrix}$. Hence there is a one-dimensional stable manifold that is tangent to the line $x+2y=0$, while the $x$-axis, which is invariant, is a one-dimensional center manifold.
Since the flow on the $x$-axis is given by $\dot{x}=-2\kappa_3 x^2$, the origin is asymptotically stable for the system restricted to the nonnegative quadrant.

Another 5 networks in \Cref{tab:fold_bimolecular} have a boundary equilibrium, namely, the origin. For the two networks in Group 1 in \Cref{tab:fold_bimolecular} with the reaction $\Y \to 2\Y$, the Jacobian matrix at the origin has two positive real eigenvalues. Finally, for the 3 networks in Group 2 in \Cref{tab:fold_bimolecular} with the reactions $\X \to \Y$, $\X \to \X+\Y$, or $\X \to 2\Y$, the origin is a saddle with its stable manifold being tangent to a line with negative slope, hence, outside the nonnegative quadrant. 

The remaining 15 networks admit no boundary equilibrium.
\end{proof}

\section{Andronov--Hopf bifurcation}
\label{sec:A-H}

Andronov--Hopf bifurcation is a codimension-one bifurcation occurring on a two-dimensional center manifold but ruled out in two-species bimolecular networks (see the proof of \cite[Theorem 4.1]{boros:hofbauer:2022}). Hence, we focus on quadratic, {\em trimolecular} $(2,m,2)$ networks. However, such networks with $m=3$ do not admit Andronov--Hopf bifurcation \cite[Sections 4.5 and 4.6]{banaji:boros:hofbauer:2024a}, and we are led to study the case $m=4$. (But note that a {\em vertical} Andronov--Hopf bifurcation occurs already in three-reaction, quadratic, trimolecular mass-action networks provided we allow more than two species \cite[Theorem 4 (III)]{banaji:boros:hofbauer:2024a}.) In \Cref{lem:946} we determine the base set of networks satisfying the necessary conditions for the existence of an isolated periodic orbit or an Andronov--Hopf bifurcation listed in \Cref{lem:mixed_source,lem:2Xto3X,lem:nitecki}. Then, in \Cref{lem:198} we enumerate the networks that admit a positive equilibrium with a pair of purely imaginary eigenvalues. It turns out that all of these networks admit an Andronov--Hopf bifurcation of some kind, and in \Cref{thm:A-H} we describe these bifurcations. 


\begin{lemma}\label{lem:946}
There are 946 dynamically nonequivalent, dynamically nontrivial, quadratic, trimolecular $(2,4,2)$ mass-action networks with the reactant complexes not being on a line, having the reactant complex $\X+\Y$, having the reaction $2\X \to 3\X$ or $2\Y \to 3\Y$, having no redundant reactions, and admitting a positive nondegenerate equilibrium.
\end{lemma}

\begin{lemma}\label{lem:198}
Out of the 946 networks in \Cref{lem:946}, 198 admit a positive equilibrium with a pair of purely imaginary eigenvalues, see \Cref{tab:A-H}.
\end{lemma}

\begin{table}[h!t]
\centering
\begin{tikzpicture}
\node[below] at (0,0) {\footnotesize
\begin{NiceTabular}{|cr@{\hspace{4pt}$\to$\hspace{4pt}}l| *{8}{w{c}{1mm}}|}
\hline
\multicolumn{3}{c}{\begin{tikzpicture}
    \node at (0,0.5) {Group 1};
    \node[rounded corners,fill=fillbg2,draw] at (0,0) {\begin{tabular}{r@{\hspace{4pt}$\to$\hspace{4pt}}l}
$2\X$ & $3\X$ \\ $\X+\Y$ & $\0$
\end{tabular}};
\end{tikzpicture}} &
\RowStyle{\rotate} $\X\to2\X$ & $\0 \to  \X$ & $\0 \to 2\X+\Y$ & $\Y \to  \0$ & $\Y \to  \X$ & $\Y \to 2\X$ & $\Y \to 3\X$ & $\Y \to  \X+\Y$ \\
\hline
&  $\X$ & $2\Y$         &        &        &        &        &        & \obllt & \obllt &        \\
&  $\X$ & $3\Y$         &        &        &        &        & \obllt & \obllt & \obllt &        \\
&  $\X$ & $\X+\Y$  &        &        &        &        & \obllt & \obllt & \obllt &        \\
& $2\X$ & $\Y$          & \bbllt & \gbllt & \gbllt & \gbllt & \gbllt & \gbllt & \gbllt & \gbllt \\
& $2\X$ & $2\Y$         & \bbllt & \gbllt & \gbllt & \gbllt & \gbllt & \gbllt & \gbllt & \gbllt \\
& $2\X$ & $3\Y$         & \bbllt & \gbllt & \gbllt & \gbllt & \gbllt & \gbllt & \gbllt & \gbllt \\
& $2\X$ & $\X+2\Y$ & \bbllt & \gbllt & \gbllt & \gbllt & \gbllt & \gbllt & \gbllt & \gbllt \\
& $2\X$ & $2\X+\Y$ & \bbllt & \gbllt & \gbllt & \gbllt & \gbllt & \gbllt & \gbllt & \gbllt \\
\hline
\end{NiceTabular}
};

\node[below] at (8.5,0.225) {\footnotesize
\begin{NiceTabular}{|cr@{\hspace{4pt}$\to$\hspace{4pt}}l| w{c}{1mm} *{3}{w{c}{4mm}} *{5}{w{c}{1mm}}|}
\hline
\multicolumn{3}{c}{\begin{tikzpicture}
    \node at (0,0.5) {Group 2};
    \node[rounded corners,fill=fillbg2,draw] at (0,0) {\begin{tabular}{r@{\hspace{4pt}$\to$\hspace{4pt}}l}
$2\X$ & $3\X$ \\ $\X+\Y$ & $\Y$
\end{tabular}};
\end{tikzpicture}} &
\RowStyle{\rotate} $\Y \to  \0$ &  $\Y \to  \X$ &  $\Y \to 2\X$ &  $\Y \to 3\X$ & $2\Y \to  \0$ & $2\Y \to  \X$ & $2\Y \to 2\X$ & $2\Y \to 3\X$ & $2\Y \to 2\X+\Y$ \\
\hline
&  $\X$ & $\Y$          &        &        & \obllt & \obllt & & & & & \\
&  $\X$ & $2\Y$         &        & \obllt & \obllt & \obllt & & & & & \\
&  $\X$ & $3\Y$         &        & \obllt & \obllt & \obllt & & & & & \\
&  $\X$ & $\X+\Y$  &        & \obllt & \obllt & \obllt & & & & & \\
&  $\X$ & $2\X+\Y$ & \bbllt & \gbllt \bbllt \obllt & \gbllt \bbllt \obllt & \gbllt \bbllt \obllt & \obllt & \obllt & \obllt & \obllt & \obllt \\
& $2\X$ & $\Y$          & \gbllt & \gbllt & \gbllt & \gbllt & & & & & \\
& $2\X$ & $2\Y$         & \gbllt & \gbllt & \gbllt & \gbllt & & & & & \\
& $2\X$ & $3\Y$         & \gbllt & \gbllt & \gbllt & \gbllt & & & & & \\
& $2\X$ & $\X+2\Y$ & \gbllt & \gbllt & \gbllt & \gbllt & & & & & \\
& $2\X$ & $2\X+\Y$ & \gbllt & \gbllt & \gbllt & \gbllt & & & & & \\
\hline
\end{NiceTabular}
};

\node[below] at (4.61,-6.18) {\footnotesize
\begin{NiceTabular}{|w{c}{4mm}r@{\hspace{4pt}$\to$\hspace{4pt}}l| *{26}{w{c}{1mm}}|}
\hline
\multicolumn{3}{c}{\begin{tikzpicture}
    \node at (0,0.5) {Group 3};
    \node[rounded corners,fill=fillbg2,draw] at (0,0) {\begin{tabular}{r@{\hspace{4pt}$\to$\hspace{4pt}}l}
$2\X$ & $3\X$ \\ $\X+\Y$ & $b\Y$ \\
\multicolumn{2}{c}{($b=2$ or $3$)}
\end{tabular}};
\end{tikzpicture}} &
\RowStyle{\rotate} $\0 \to  \X$ & $\0 \to 2\X+\Y$ & $\0 \to \X+\Y$ & $\0 \to \X+2\Y$ & $\0 \to \Y$ & $\X \to 2\X$ & $\X \to 2\X+\Y$ & $\X \to \X+\Y$ & $\X \to 3\Y$ & $\X \to 2\Y$ & $\X \to \Y$ & $\Y \to \X$ & $\Y \to 2\X$ & $\Y \to 3\X$ & $\Y \to \X+\Y$ & $\Y \to \X+2\Y$ & $2\X \to 2\X+\Y$ & $2\X \to \X+2\Y$ & $2\X \to 3\Y$ & $2\X \to 2\Y$ & $2\X \to \Y$ & $2\Y \to  \0$ & $2\Y \to \X$ & $2\Y \to 2\X$ & $2\Y \to 3\X$ & $2\Y \to 2\X+\Y$ \\
\hline
& $\Y$ & $\0$ & \gbllt & \gbllt & \gbllt & \gbllt & \gbllt & & \gbllt & \gbllt & \gbllt & \gbllt & \gbllt & \gbllt & \gbllt & \gbllt & \gbllt & \gbllt & \gbllt & \gbllt & \gbllt & \gbllt & \gbllt & \bbllt & \gbllt & \gbllt & \gbllt & \gbllt \\
& $2\Y$ & $\0$ & \bbllt & \obllt & \obllt & \obllt & & \bbllt & \obllt &        &        &        &         & \gbllt & \gbllt & \gbllt & \bbllt & \obllt &        &        &         &        &        &         & & & & \\
\hline
\end{NiceTabular}
};

\node[below] at (-1.5,-9.5) {\footnotesize
\begin{NiceTabular}{|w{c}{5mm}r@{\hspace{4pt}$\to$\hspace{4pt}}l| w{c}{1mm}|}
\hline
\multicolumn{3}{c}{\begin{tikzpicture}
    \node at (0,0.5) {Group 4};
    \node[rounded corners,fill=fillbg2,draw] at (0,0) {\begin{tabular}{r@{\hspace{4pt}$\to$\hspace{4pt}}l}
$2\X$ & $3\X$ \\ $\X+\Y$ & $\X$
\end{tabular}};
\end{tikzpicture}} &
\RowStyle{\rotate} $\Y\to\X+2\Y$ \\
\hline
&  $\X$ &  $\0$ & \bbllt \\
&  $\X$ &  $\Y$ & \obllt \\
&  $\X$ & $2\Y$ & \obllt \\
&  $\X$ & $3\Y$ & \obllt \\
\hline
\end{NiceTabular}
};

\node[below] at (3.6,-9.5) {\footnotesize
\begin{NiceTabular}{|w{c}{5mm}r@{\hspace{4pt}$\to$\hspace{4pt}}l| *{5}{w{c}{1mm}}|}
\hline
\multicolumn{3}{c}{\begin{tikzpicture}
    \node at (0,0.5) {Group 5};
    \node[rounded corners,fill=fillbg2,draw] at (0,0) {\begin{tabular}{r@{\hspace{4pt}$\to$\hspace{4pt}}l}
$2\X$ & $3\X$ \\ $\X+\Y$ & $2\X$
\end{tabular}};
\end{tikzpicture}} &
\RowStyle{\rotate} $\0 \to           \Y$ & $\0 \to \X+2\Y$ & $\0 \to \X+\Y$ & $\0 \to 2\X+\Y$ & $\Y \to  \X+2\Y$ \\
\hline
&  $\X$ &  $\0$ & \bbllt & \obllt & \obllt & \obllt & \gbllt \\
\hline
\end{NiceTabular}
};

\node[below] at (9.7,-9.5) {\footnotesize
\begin{NiceTabular}{|w{c}{5mm}r@{\hspace{4pt}$\to$\hspace{4pt}}l| *{5}{w{c}{1mm}}|}
\hline
\multicolumn{3}{c}{\begin{tikzpicture}
    \node at (0,0.5) {Group 6};
    \node[rounded corners,fill=fillbg2,draw] at (0,0) {\begin{tabular}{r@{\hspace{4pt}$\to$\hspace{4pt}}l}
$2\X$ & $3\X$ \\ $\X+\Y$ & $3\X$
\end{tabular}};
\end{tikzpicture}} &
\RowStyle{\rotate} $\0 \to           \Y$ & $\0 \to \X+2\Y$ & $\0 \to \X+\Y$ & $\0 \to 2\X+\Y$ & $\Y \to  \X+2\Y$ \\
\hline
&  $\X$ &  $\0$ & \bbllt & \pbllt & \obllt & \obllt & \gbllt \\
\hline
\end{NiceTabular}
};

\node[below] at (4.61,-12.25) {\footnotesize
\begin{NiceTabular}{|cr@{\hspace{4pt}$\to$\hspace{4pt}}l| *{3}{w{c}{1mm}}|}
\hline
\multicolumn{3}{c}{\begin{tikzpicture}
    \node at (0,0.5) {Group 7};
    \node[rounded corners,fill=fillbg2,draw] at (0,0) {\begin{tabular}{r@{\hspace{4pt}$\to$\hspace{4pt}}l}
$2\X$ & $3\X$ \\ $\Y$ & $a\X$ \\
\multicolumn{2}{c}{($a=1$, $2$, or $3$)}
\end{tabular}};
\end{tikzpicture}} &
\RowStyle{\rotate} $\X+\Y \to \X$ & $\X+\Y \to \0$ & $\X+\Y \to \Y$ \\
\hline
&  $\X+\Y$ &  $2\Y$ & \gbllt & \gbllt & \gbllt \\
&  $\X+\Y$ &  $3\Y$ & \gbllt & \gbllt & \gbllt \\
&  $\X+\Y$ &  $\X+2\Y$ & & \gbllt & \gbllt \\
\hline
\end{NiceTabular}
};

\end{tikzpicture}
\caption{The 198 dynamically nontrivial, quadratic, trimolecular $(2,4,2)$ mass-action networks that admit an Andronov--Hopf bifurcation. The networks are partitioned into 7 groups. In each group, the first and the second reactions are the same for all networks (indicated in the top left corner), while the third and the fourth reactions are the row- and the column headers, respectively. In Group 7, the complex $\X+\Y$ is the source of two reactions, while it is the source of only one reaction for the networks in Groups 1--6. A green \gbllt, a blue \bbllt, or an orange \obllt\, symbol indicates that the Andronov--Hopf bifurcation is \mygreen{supercritical}, \blue{vertical}, or \orange{subcritical}, respectively. Notice that all three types of bifurcations occur for 3 networks in Group 2. The single network that admits a Bautin bifurcation is marked with the purple symbol \pbllt; the second focal value is positive. Networks without any symbol do not admit a positive equilibrium with a pair of purely imaginary eigenvalues.}
\label{tab:A-H}
\end{table}

The transversality condition of the Andronov--Hopf bifurcation (i.e., that the eigenvalues cross the imaginary axis with nonzero speed \cite[Theorem 3.3 (B.2)]{kuznetsov:2023}) is easily verified for all 198 networks in \Cref{lem:198}. To decide whether the bifurcation is supercritical, subcritical, or degenerate, one computes the first focal value, $L_1$. At the critical parameter value, the stability of the equilibrium is determined by the sign of the first nonzero focal value. Hence, when $L_1$ vanishes, one calculates the further focal values ($L_2, L_3, \ldots$) until a nonzero one is found. When $L_k=0$ for all $k\geq1$, the equilibrium is a center, and the Andronov--Hopf bifurcation is termed \emph{vertical}. For planar quadratic systems, an equilibrium with a pair of purely imaginary eigenvalues and $L_1=L_2=L_3=0$ is a center, see (the proof of) the Kapteyn--Bautin Theorem in \cite[Section 8.7]{dumortier:llibre:artes:2006}. The following theorem summarises the result of the analysis of the focal values in the 198 networks in \Cref{tab:A-H}. All the computations can be found in the Mathematica Notebook \cite{boros_github}.


\begin{theorem} \label{thm:A-H}
The analysis of the 198 networks in \Cref{tab:A-H} leads to the following classification.
\begin{itemize}
    \item $L_1<0$ for the 135 networks marked with \mygreen{$-$}, and hence, the Andronov--Hopf bifurcation is supercritical.
    \item $L_1>0$ for the 42 networks marked with \orange{$+$}, and hence, the Andronov--Hopf bifurcation is subcritical.
    \item $L_1=L_2=L_3=0$  for the 17 networks marked with \blue{$0$}, and hence, the Andronov--Hopf bifurcation is vertical.
    \item $L_1$ can have any sign and $L_1=0$ implies $L_2=L_3=0$ for the 3 networks marked with \mygreen{$-$}\blue{$0$}\orange{$+$}, and hence, supercritical, vertical, and subcritical Andronov--Hopf bifurcations are all admitted.
    \item $L_1$ can have any sign and $L_1=0$ implies $L_2>0$ for the network marked with \purple{B}, and hence, the equilibrium undergoes a subcritical Bautin bifurcation.
\end{itemize}
\end{theorem}

\begin{remark} \label{rem:A-H}
We conclude this section with several comments.
\begin{enumerate}[(a)]
\item The 198 dynamically nonequivalent networks in \Cref{tab:A-H} fall into 157 diagonally nonequivalent classes \cite{boros_github}.

\item All the 198 networks in \Cref{tab:A-H} are atoms of Andronov--Hopf bifurcation (see \Cref{subsec:inheritance}). In fact, these are the \emph{smallest} planar, quadratic, trimolecular mass-action networks that admit an Andronov--Hopf bifurcation in the sense that no such network with only three reactions exists \cite[Theorem 4]{banaji:boros:hofbauer:2024a}. However, we note that the same theorem also implies that there exists a (unique) quadratic, trimolecular mass-action network with {\em three species} and three reactions that admits a vertical Andronov--Hopf bifurcation.

\item\label{item:bautin} The single network for which $L_2\neq0$ holds whenever $L_1=0$ is
\begin{center}
\begin{tabular}{*{4}{r@{\hspace{4pt}$\to$\hspace{4pt}}l}}
$2\X$ & $3\X$ & $\X+\Y$ & $3\X$ & $\X$ & $\0$ & $\0$ & $\X+2\Y$ \\
\end{tabular}
\end{center}
This network displays a Bautin bifurcation \cite[Section 8.3]{kuznetsov:2023}. In our case, $L_2>0$ at the bifurcation, and consequently the equilibrium is repelling when $L_1=0$. By varying the rate constants, two small limit cycles are born (a repelling equilibrium is surrounded by an attracting and a repelling limit cycle). For nearby rate constants, a fold bifurcation of limit cycles also occurs.

\item\label{item:frank-kamenetsky-salnikov} The quadratic, trimolecular $(2,5,2)$ network
\begin{center}
\begin{tabular}{lll}
$\X \to 2\X \to 3\X$ & $\X+\Y \to 2\Y$ & $\Y \rightleftarrows \0$
\end{tabular}
\end{center}
was studied by Frank-Kamenetsky and Salnikov \cite{frank-kamenetsky:salnikov:1943}. They showed that the network admits an Andronov--Hopf bifurcation, which is always supercritical. In fact, this network is obtained from a network in Group 3 in \Cref{tab:A-H} by adding the reaction $\X \to 2\X$, i.e., by the application of enlargement E1. Hence, the supercritical Andronov--Hopf bifurcation observed by Frank-Kamenetsky and Salnikov is itself inherited from the smaller network -- see \Cref{lem:inherit}. We return to both of these networks later and observe that both, in fact, admit Bogdanov--Takens bifurcations.

\newpage
\item For the 20 networks that admit a vertical Andronov--Hopf bifurcation, we also prove directly that they indeed admit a center, see \Cref{lem:A-H_vertical} in \Cref{sec:A-H_vertical}. Here we remark that, apart from these 20 networks, there are further quadratic, trimolecular $(2,4,2)$ networks that admit a center. For example, the two networks
\begin{center}
\begin{tabular}{*{4}{r@{\hspace{4pt}$\to$\hspace{4pt}}l}}
\arrayrulecolor{mygray}
$\X$ & $2\X$ & $\X$ & $\0$ & $\X+\Y$ & $2\Y$ & $\Y$ & $\0$ \\
\hline
\multicolumn{8}{c}{} \\[-10pt]
$\X$ & $2\X$ & $\X+\Y$ & $\X+2\Y$& $\X+\Y$ & $\0$ & $\Y$ & $\0$ \\
\end{tabular}
\end{center}
lead to the classical Lotka--Volterra differential equation, and whenever a positive equilibrium exists, it is a global center in the positive quadrant. These networks have neither of the reactions $2\X \to 3\X$ or $2\Y \to 3\Y$, and thus the eigenvalues cannot cross the imaginary axis while the rate constants are varied, ruling out an Andronov--Hopf bifurcation. Considering only networks without redundant reactions (see \Cref{sec:isomorph}), there are 21 dynamically nonequivalent, quadratic, trimolecular $(2,4,2)$ networks with the properties described above. For the complete list, see \cite{boros_github}. 

\item Out of the 198 networks in \Cref{tab:A-H}, 104 have only three distinct reactant complexes. These 104 networks admit only supercritical or vertical Andronov--Hopf bifurcation. This is consistent with the fact that a subcritical Andronov--Hopf bifurcation is forbidden for any planar, quadratic network with three reactions, see \cite[Theorem 3]{banaji:boros:hofbauer:2024a}. The 104 networks here are related to the networks in that theorem as follows: the 3 distinct reactant complexes are
\begin{itemize}
    \item $2\X$, $\X+\Y$, $\X$ in 5 of the 104 networks and the Andronov--Hopf bifurcation is vertical, cf.\ Case 8;
    \item $2\X$, $\X+\Y$, $\Y$ in 89 of the 104 networks and the Andronov--Hopf bifurcation is supercritical, cf.\ Case 9;
    \item $2\X$, $\X+\Y$, $\0$ in 10 of the 104 networks and the Andronov--Hopf bifurcation is supercritical, cf.\ Case 10.
\end{itemize}

\item We remark that a large fraction of the 198 networks in \Cref{tab:A-H} are obtained from the generalised LVA network \eqref{eq:LVA} by enlargement E1: indeed, 50 networks in Group 3 are such (48 admitting a supercritical Andronov--Hopf bifurcation, while 2 admitting a vertical Andronov--Hopf bifurcation). Though the generalised LVA network does not admit a periodic orbit (it has a unique positive equilibrium that is a global repellor), several of its modifications apparently do admit an Andronov--Hopf bifurcation, including the Frank-Kamenetsky--Salnikov network discussed in \eqref{item:frank-kamenetsky-salnikov} above.

\end{enumerate}
\end{remark}
\section{Bogdanov--Takens bifurcation}
\label{sec:B-T}

Consider a two-parameter family of planar differential equations. An equilibrium of such a system whose Jacobian matrix is nilpotent with index two undergoes a Bogdanov--Takens bifurcation, provided some nondegeneracy and transversality conditions are fulfilled \cite[Section 8.4]{kuznetsov:2023}. In a neighbourhood of the critical parameter value of this codimension-two bifurcation, one finds the following codimension-one bifurcations: Andronov--Hopf, homoclinic, and fold bifurcations. Consequently, a limit cycle, a homoclinic orbit, and multiple isolated equilibria are all possible in such a family of differential equations.

Our main goal in this section is to find two-species, quadratic, trimolecular mass-action networks that exhibit a Bogdanov--Takens bifurcation. By \Cref{thm:nplus2sources}, any two-species, rank-two mass-action network with multiple positive nondegenerate equilibria has at least four distinct reactant complexes, and hence, at least four reactions. Therefore we will search for Bogdanov--Takens bifurcation in networks in quadratic, trimolecular $(2,4,2)$ networks. However, we remark that if we allow more than two species, a nondegenerate Bogdanov--Takens bifurcation occurs already in three-reaction networks, see for example
\begin{center}
\begin{tabular}{*{3}{r@{\hspace{4pt}$\to$\hspace{4pt}}l}l}
\arrayrulecolor{mygray}
$2\X$ & $4\X+3\Y+\Z$ & $\X+\Y$ & $\msf{0}$ & $\Z$ & $\X$ & (quadratic, octomolecular)\\
\hline
\multicolumn{7}{c}{} \\[-10pt]
$2\X$ & $3\X$ & $\X+\Y+\Z$ & $2\Y$ & $\Y$ & $\Z$ & (trimolecular) \\
\end{tabular}
\end{center}
The analysis of these two networks can be found at \cite{boros_github}.

From now on, we focus on finding \emph{all} quadratic, trimolecular $(2,4,2)$ mass-action networks that admit a Bogdanov--Takens bifurcation. The enumeration of the set of networks appearing in the following lemma is performed in \cite{boros_github}.

\begin{lemma} \label{lem:B-T-candidates}
The 831 networks that admit a nondegenerate fold bifurcation enumerated in \Cref{thm:834_quadratic_trimolecular_fold} and the 198 networks that admit an Andronov--Hopf bifurcation enumerated in \Cref{thm:A-H} have 40 networks in common. Out of these 40 networks, 33 admit a positive equilibrium with a double zero eigenvalue, see \Cref{tab:B-T}.
\end{lemma}

\begin{table}[h!t]
\centering
{\footnotesize
\begin{tabular}{cr|*{4}{r@{\hspace{4pt}$\to$\hspace{4pt}}l}}

\multirow{8}{*}{\mygreen{supercritical B--T}}
& 1 & $2\X$ & $3\X$ & $\X+\Y$ & $2\Y$ & $\Y$ & $\0$ & $\0$ & $\Y$ \\
& 2 & $2\X$ & $3\X$ & $\X+\Y$ & $3\Y$ & $\Y$ & $\0$ & $\0$ & $\Y$ \\
& 3 & $2\X$ & $3\X$ & $\X+\Y$ & $2\Y$ & $\Y$ & $\0$ & $\X$ & $\Y$ \\
& 4 & $2\X$ & $3\X$ & $\X+\Y$ & $3\Y$ & $\Y$ & $\0$ & $\X$ & $\Y$ \\
& 5 & $2\X$ & $3\X$ & $\X+\Y$ & $2\Y$ & $\Y$ & $\0$ & $\X$ & $2\Y$ \\
& 6 & $2\X$ & $3\X$ & $\X+\Y$ & $3\Y$ & $\Y$ & $\0$ & $\X$ & $2\Y$ \\
& 7 & $2\X$ & $3\X$ & $\X+\Y$ & $2\Y$ & $\Y$ & $\0$ & $\X$ & $3\Y$ \\
& 8 & $2\X$ & $3\X$ & $\X+\Y$ & $3\Y$ & $\Y$ & $\0$ & $\X$ & $3\Y$ \\
&&\multicolumn{8}{c}{}\\[-8pt]
\hline
&&\multicolumn{8}{c}{}\\[-8pt]
\multirow{2}{*}{\blue{vertical B--T}}
&  9 & $2\X$ & $3\X$ & $\X+\Y$ & $2\X$ & $\0$ & $\Y$ & $\X$ & $\0$ \\
& 10 & $2\X$ & $3\X$ & $\X+\Y$ & $3\X$ & $\0$ & $\Y$ & $\X$ & $\0$\\
&&\multicolumn{8}{c}{}\\[-8pt]
\hline
&&\multicolumn{8}{c}{}\\[-8pt]
\multirow{23}{*}{\orange{subcritical B--T}}
& 11 & $2\X$ & $3\X$ & $\X+\Y$ & $2\X$ & $\0$ &  $\X+2\Y$ & $\X$ & $\0$ \\
& 12 & $2\X$ & $3\X$ & $\X+\Y$ & $3\X$ & $\0$ &  $\X+2\Y$ & $\X$ & $\0$ \\
& 13 & $2\X$ & $3\X$ & $\X+\Y$ & $2\X$ & $\0$ &  $\X+ \Y$ & $\X$ & $\0$ \\
& 14 & $2\X$ & $3\X$ & $\X+\Y$ & $3\X$ & $\0$ &  $\X+ \Y$ & $\X$ & $\0$ \\
& 15 & $2\X$ & $3\X$ & $\X+\Y$ & $2\X$ & $\0$ & $2\X+ \Y$ & $\X$ & $\0$ \\
& 16 & $2\X$ & $3\X$ & $\X+\Y$ & $3\X$ & $\0$ & $2\X+ \Y$ & $\X$ & $\0$ \\

& 17 & $2\X$ & $3\X$ & $\X+\Y$ & $\X$ & $\Y$ & $\X+2\Y$ & $\X$ &  $\Y$ \\
& 18 & $2\X$ & $3\X$ & $\X+\Y$ & $\X$ & $\Y$ & $\X+2\Y$ & $\X$ & $2\Y$ \\
& 19 & $2\X$ & $3\X$ & $\X+\Y$ & $\X$ & $\Y$ & $\X+2\Y$ & $\X$ & $3\Y$ \\

& 20 & $2\X$ & $3\X$ & $\X+\Y$ & $\0$ & $\Y$ & $2\X$ & $\X$ & $2\Y$ \\
& 21 & $2\X$ & $3\X$ & $\X+\Y$ & $\0$ & $\Y$ & $3\X$ & $\X$ & $2\Y$ \\
& 22 & $2\X$ & $3\X$ & $\X+\Y$ & $\0$ & $\Y$ &  $\X$ & $\X$ & $3\Y$ \\
& 23 & $2\X$ & $3\X$ & $\X+\Y$ & $\0$ & $\Y$ & $2\X$ & $\X$ & $3\Y$ \\
& 24 & $2\X$ & $3\X$ & $\X+\Y$ & $\0$ & $\Y$ & $3\X$ & $\X$ & $3\Y$ \\
& 25 & $2\X$ & $3\X$ & $\X+\Y$ & $\0$ & $\Y$ &  $\X$ & $\X$ & $\X+\Y$ \\
& 26 & $2\X$ & $3\X$ & $\X+\Y$ & $\0$ & $\Y$ & $2\X$ & $\X$ & $\X+\Y$ \\
& 27 & $2\X$ & $3\X$ & $\X+\Y$ & $\0$ & $\Y$ & $3\X$ & $\X$ & $\X+\Y$ \\

& 28 & $2\X$ & $3\X$ & $\X+\Y$ &  $\Y$ & $2\Y$ & $\0$     & $\X$ & $2\X+\Y$ \\
& 29 & $2\X$ & $3\X$ & $\X+\Y$ &  $\Y$ & $2\Y$ & $\X$     & $\X$ & $2\X+\Y$ \\
& 30 & $2\X$ & $3\X$ & $\X+\Y$ &  $\Y$ & $2\Y$ & $2\X$    & $\X$ & $2\X+\Y$ \\
& 31 & $2\X$ & $3\X$ & $\X+\Y$ &  $\Y$ & $2\Y$ & $3\X$    & $\X$ & $2\X+\Y$ \\
& 32 & $2\X$ & $3\X$ & $\X+\Y$ &  $\Y$ & $2\Y$ & $2\X+\Y$ & $\X$ & $2\X+\Y$ \\

& 33 & $2\X$ & $3\X$ & $\X+\Y$ & $2\Y$ & $2\Y$ & $\0$     & $\0$ & $\X+2\Y$ \\
\end{tabular}
}
\caption{The 33 dynamically nonequivalent, quadratic, trimolecular $(2,4,2)$ mass-action networks that admit a Bogdanov--Takens (B--T) bifurcation.}
\label{tab:B-T}
\end{table}

At the end of this section (in \Cref{rem:B-T} \eqref{item:hopf_fold_but_no_B-T}), we comment on the 7 mass-action networks that admit both a fold and an Andronov--Hopf bifurcation but not a Bogdanov--Takens bifurcation. The 33 networks in \Cref{lem:B-T-candidates} are the candidates for Bogdanov--Takens bifurcation among the quadratic, trimolecular $(2,4,2)$ networks. We arrived at the same 33 networks in \Cref{thm:834_quadratic_trimolecular_fold} by requiring a positive equilibrium with one eigenvalue which changes sign, while the other is zero; it so happens that all of these networks also admit a pair of purely imaginary eigenvalues.




For the 33 networks in \Cref{lem:B-T-candidates} (or \Cref{tab:B-T}), a nondegenerate Bogdanov--Takens bifurcation occurs and is unfolded transversely by the rate constants whenever the conditions (BT.0), (BT.1), (BT.2) and (BT.3) in \cite[Theorem 8.4]{kuznetsov:2023} are fulfilled. Condition (BT.0) is the basic assumption that the equilibrium, at the critical parameter value, has a nonzero Jacobian matrix with a double zero eigenvalue with geometric multiplicity 1, i.e., the Jacobian matrix is nilpotent with index two. 
Then after a suitable linear change of variables, the Taylor expansion of the vector field  can be written as
\begin{align*}
    \begin{split}
        \dot{x} &= y + \frac12 a_{20}x^2 + \cdots, \\
        \dot{y} &=  \frac12 b_{20} x^2 + b_{11} xy+ \cdots.
    \end{split}
\end{align*}
The nondegeneracy conditions (BT.1) and (BT.2) are $a_{20} + b_{11} \neq 0$ and $b_{20} \neq 0$, respectively. If both (BT.1) and (BT.2) hold, the sign of $(a_{20} + b_{11})b_{20}$, denoted by $\sigma$, determines the (truncated) normal form of the bifurcation. If the transversality condition (BT.3) is also satisfied then the local dynamics is fully captured by the truncated normal form 
\begin{align*}
    \begin{split}
        \dot{x} &= y, \\
        \dot{y} &= \beta_1 + \beta_2 x + x^2 + \sigma xy
    \end{split}
\end{align*}
(with $\sigma = \pm 1$ and two small parameters $\beta_1$ and $\beta_2$). The bifurcation diagram in the case of $\sigma=-1$ is shown in \cite[Fig.\ 8.8]{kuznetsov:2023}; 
the interesting features include a supercritical Andronov--Hopf bifurcation, 
resulting in a stable limit cycle which grows into a homoclinic
orbit (which is attracting from the inside) and then disappears.
On the other hand, if $\sigma = +1$ then the Andronov--Hopf bifurcation is subcritical, the limit cycle is unstable, and the homoclinic orbit is repelling. With some abuse of terminology, we call the Bogdanov--Takens bifurcation {\em supercritical} if $\sigma = -1$, and {\em subcritical} if 
$\sigma = +1$.

Condition (BT.0) is easily verified for each of the 33 networks. Checking conditions (BT.1) and (BT.2) requires more effort. It turns out that 8 networks satisfy them with $\sigma=-1$ (leading to the supercritical case; networks 1--8 in \Cref{tab:B-T}), while another 23 networks also satisfy both conditions but with $\sigma=+1$ (leading to the subcritical case; networks 11--33 in \Cref{tab:B-T}). For all of these 31 networks, the transversality condition (BT.3) is also fulfilled, and hence, they all admit a nondegenerate Bogdanov--Takens bifurcation that is unfolded transversely by the rate constants. For networks 9 and 10 in \Cref{tab:B-T}, condition (BT.1) is violated: $a_{20} + b_{11} = 0$. In general, when one (or both) of (BT.1) and (BT.2) are violated, the Bogdanov--Takens bifurcation is degenerate, and this might lead to a complicated higher codimension bifurcation which is so far only partially understood. Note, however, that by \Cref{thm:bifcodim} below, bifurcations of codimension three or higher cannot be unfolded by the rate constants in a $(2,4,2)$ network. Indeed in our 2 networks, we encounter only the simple special case when the Andronov--Hopf bifurcation is vertical and occurs simultaneously with the homoclinic bifurcation. In this case, the region bounded by the homoclinic orbit is filled with periodic orbits and the equilibrium is a center, see \Cref{fig:vertBT}. With some abuse of terminology, we call this a {\em vertical Bogdanov--Takens bifurcation}.

We summarise our main findings in the following theorem; the calculations are presented in \cite{boros_github}. 
\begin{theorem}
\label{thm:BTmain}
There are 33 dynamically nonequivalent, quadratic, trimolecular $(2,4,2)$ mass-action networks which admit a Bogdanov--Takens bifurcation (see \Cref{tab:B-T}). The bifurcation is supercritical for networks 1--8, vertical for networks 9 and 10, and subcritical for networks 11--33. 
\end{theorem}

For completeness, we present a brief analysis of network 9 in \Cref{tab:B-T}, one of the networks admitting a vertical Bogdanov--Takens bifurcation (the analysis of network 10 is similar; in fact, networks 9 and 10 are diagonally equivalent). The associated mass-action differential equation reads as
\begin{align}\label{eq:ode_B-T_vertical}
    \begin{split}
        \dot{x} &= \kappa_1 x^2 + \kappa_2 xy - \kappa_3 x, \\
        \dot{y} &= -\kappa_2 xy + \kappa_4.
    \end{split}
\end{align}
By a short calculation, there are $0$, $1$, or $2$ positive equilibria if  $4\kappa_1\kappa_4 - \kappa_3^2$ is positive, zero, or negative, respectively. Furthermore, after multiplication by the Dulac function $1/x$, the divergence of the vector field is $\kappa_1 - \kappa_2$, a constant. Hence, by the Bendixson--Dulac test, there is no periodic orbit or homoclinic orbit when $\kappa_1 \neq \kappa_2$. On the other hand, for $\kappa_1 = \kappa_2$ the system is Hamiltonian with the Hamiltonian function $H(x,y)=\kappa_1 x y+\frac{\kappa_1}{2}y^2-\kappa_3 y-\kappa_4 \log x$. Hence, when $4\kappa_1\kappa_4 < \kappa_3^2$ and $\kappa_1 = \kappa_2$, the positive equilibrium with the positive Jacobian determinant is a center (it undergoes a vertical Andronov--Hopf bifurcation) and there is a homoclinic orbit (on the level set of $H$ through the saddle), whose interior is filled with periodic orbits, see \Cref{fig:vertBT}. When fixing $\kappa_3$ and $\kappa_4$, and regarding \eqref{eq:ode_B-T_vertical} as a two-parameter family of differential equations with $(\kappa_1,\kappa_2)$ being varied in a small neighbourhood of $(\frac{\kappa_3^2}{4\kappa_4},\frac{\kappa_3^2}{4\kappa_4})$, the dynamic behaviour is similar to that of a mechanical system with (possibly negative) friction. In particular, a ``normal form'' or unfolding of a vertical Bogdanov--Takens bifurcation would be given by 
\begin{align*}
    \begin{split}
        \dot{x} &= y, \\
        \dot{y} &= \beta_1 + \beta_2 y + x^2 
    \end{split}
\end{align*}
(with  two small parameters $\beta_1$ and $\beta_2$).
\begin{figure}[ht]
\centering
\begin{tabular}{cc}
\includegraphics[scale=0.4]{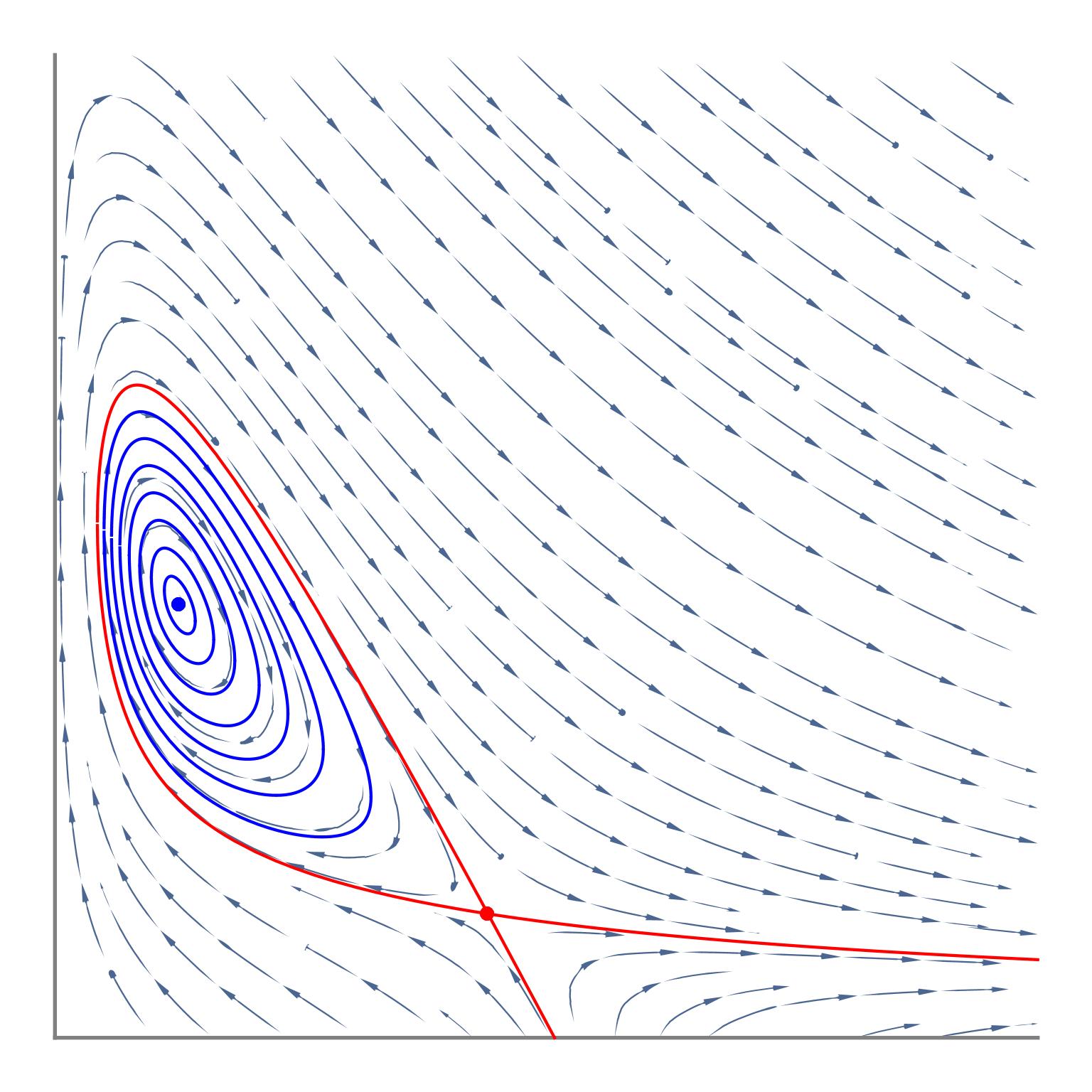} &
\includegraphics[scale=0.4]{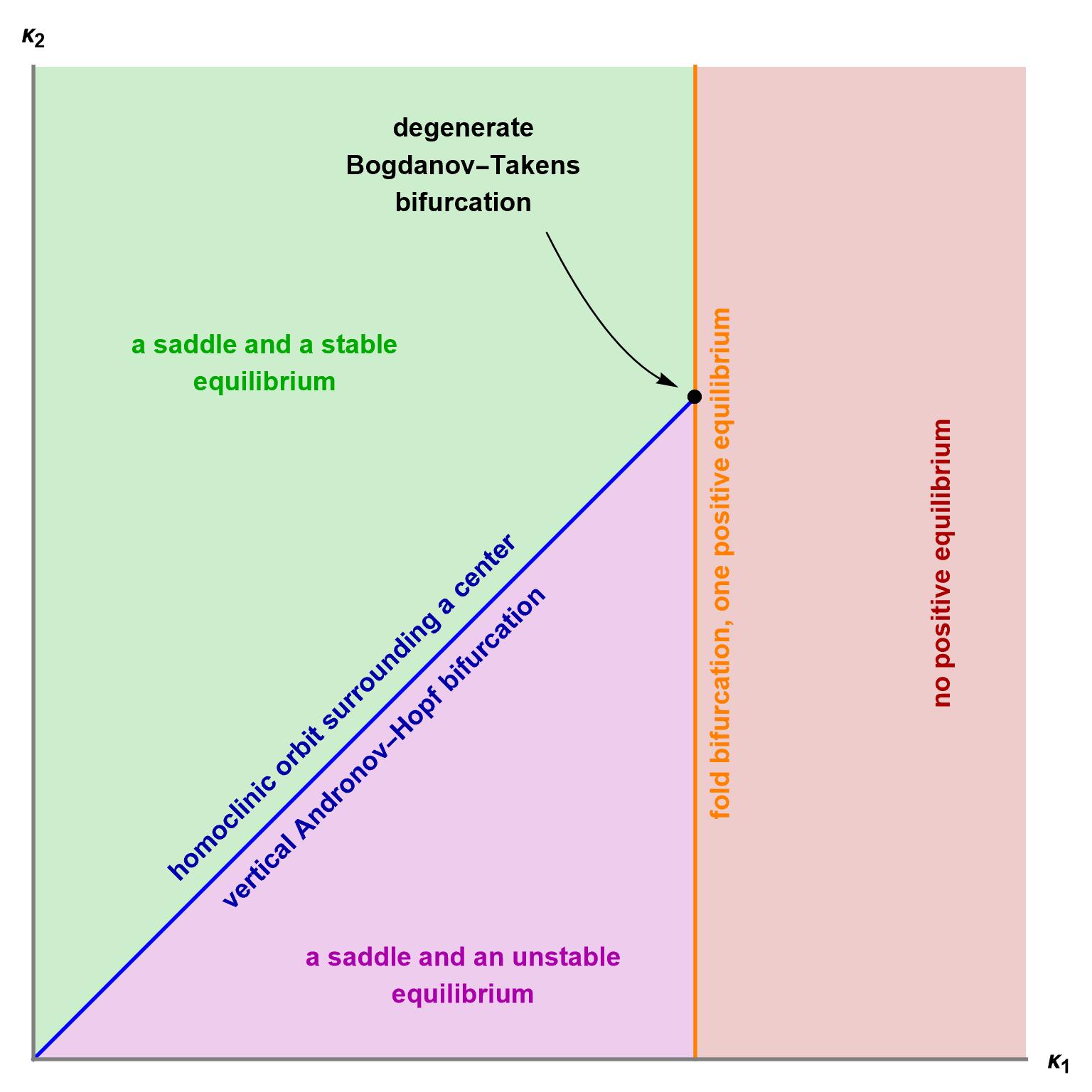}
\end{tabular}
\caption{The phase portrait of network 9 in \Cref{tab:B-T} with $4\kappa_1\kappa_4<\kappa_3^2$ and $\kappa_1=\kappa_2$ (left panel) and the bifurcation diagram with $\kappa_3, \kappa_4>0$ being fixed, while $\kappa_1, \kappa_2>0$ are parameters (right panel).}
\label{fig:vertBT}
\end{figure}

\begin{remark} \label{rem:B-T}
We conclude this section with several comments.
\begin{enumerate}[(a)]
\item The 33 networks in \Cref{tab:B-T} fall into 28 diagonally nonequivalent classes. In particular, the following five are diagonally equivalent pairs: (1,2), (3,6), (9,10), (11,14), (13,16).

\item Notice that, interestingly, each of the eight networks that lead to a supercritical Bogdanov--Takens bifurcation (networks 1--8 in \Cref{tab:B-T}) are obtained from the generalised LVA network \eqref{eq:LVA} by adding a fourth reaction. What is even more remarkable is that the Frank-Kamenetsky--Salnikov network (see \Cref{rem:A-H} \eqref{item:frank-kamenetsky-salnikov}) is obtained from network 1 in \Cref{tab:B-T} by adding the reaction $\X \to 2\X$, and thus the supercritical Bogdanov--Takens bifurcation in the smaller network is inherited by the larger network, see \Cref{lem:inherit}. It does not follow from \Cref{lem:inherit} that the Bogdanov--Takens bifurcation in the Frank-Kamenetsky--Salnikov network must be everywhere supercritical, but a direct calculation shows that this is the case.

\item As discussed in \Cref{subsec:fold_planar}, the quadratic, trimolecular $(2,4,2)$ networks in \Cref{rem:fold_quadratic} \eqref{item:fold_quadratic_2d_center_manifold} admit a positive equilibrium with a nilpotent Jacobian matrix of index two; however, a pair of purely imaginary eigenvalues is forbidden, and hence, no Andronov--Hopf bifurcation occurs. Consequently, a Bogdanov--Takens bifurcation cannot be unfolded transversely by the rate constants in these networks.

\item In \Cref{subsec:fold_planar} we identified fifteen $(2,4,2)$ networks (see \Cref{rem:fold_quadratic} \eqref{item:fold_quadratic_E6}), where the fold bifurcation is inherited from a $(1,3,1)$ network via enlargement E6 (see \Cref{lem:inherit}). Interestingly, two of the fifteen networks even admit a nondegenerate Bogdanov--Takens bifurcation: networks 13 and 14 in \Cref{tab:B-T} appear with $C=\X+\Y$ in the second and the third row in \Cref{rem:fold_quadratic} \eqref{item:fold_quadratic_E6}, respectively.

\item Recall from \Cref{sec:A-H} that there is a single quadratic, trimolecular $(2,4,2)$ network that admits a (subcritical) Bautin bifurcation. The same network happens to admit a (subcritical) Bogdanov--Takens bifurcation, see network 12 in \Cref{tab:B-T}. However, the two codimension-two bifurcations occur in separate parts of the parameter space. Nevertheless, this network is rather special, in that it admits an array of interesting behaviours including both multiple limit cycles and a homoclinic orbit. 

\item \label{item:hopf_fold_but_no_B-T} In \Cref{lem:B-T-candidates}, we have identified 7 networks (listed below) that admit both an Andronov--Hopf and a fold bifurcation, but not a double zero eigenvalue: the two bifurcation sets are separated, and hence, these networks do not exhibit a Bogdanov--Takens bifurcation. Interestingly, each network is obtained by supplementing the generalised LVA \eqref{eq:LVA} by a reaction with reactant complex $2\Y$ (they appear in Group 3 in \Cref{tab:A-H}). It is an open question whether homoclinic bifurcations can occur in these networks. Notice that these 7 networks fall into 5 diagonally nonequivalent classes: networks 1 \& 4 and networks 3 \& 5 are diagonally equivalent pairs.
\begin{center}
\begin{tabular}{r|*{4}{r@{\hspace{4pt}$\to$\hspace{4pt}}l}}
\arrayrulecolor{mygray}
1 & $2\X$ & $3\X$ & $\X+\Y$ & $2\Y$ & $\Y$ & $\0$ & $2\Y$ & $2\X$ \\
2 & $2\X$ & $3\X$ & $\X+\Y$ & $2\Y$ & $\Y$ & $\0$ & $2\Y$ & $3\X$ \\
3 & $2\X$ & $3\X$ & $\X+\Y$ & $2\Y$ & $\Y$ & $\0$ & $2\Y$ & $2\X+\Y$ \\
4 & $2\X$ & $3\X$ & $\X+\Y$ & $3\Y$ & $\Y$ & $\0$ & $2\Y$ & $\X$ \\
5 & $2\X$ & $3\X$ & $\X+\Y$ & $3\Y$ & $\Y$ & $\0$ & $2\Y$ & $2\X$ \\
6 & $2\X$ & $3\X$ & $\X+\Y$ & $3\Y$ & $\Y$ & $\0$ & $2\Y$ & $3\X$ \\
7 & $2\X$ & $3\X$ & $\X+\Y$ & $3\Y$ & $\Y$ & $\0$ & $2\Y$ & $2\X+\Y$ \\
\end{tabular}
\end{center}

\end{enumerate}
\end{remark}

\section{Other bifurcations of codimension 2 and higher}
\label{sec:other}

We have, so far, fully characterised fold, Andronov--Hopf, and Bogdanov--Takens bifurcations in planar, quadratic mass-action networks with up to four reactions, and molecularity three or less. We have also observed that a unique network, 
\begin{center}
\begin{tabular}{*{4}{r@{\hspace{4pt}$\to$\hspace{4pt}}l}}
$2\X$ & $3\X$ & $\X+\Y$ & $3\X$ & $\X$ & $\0$ & $\0$ & $\X+2\Y$ \\
\end{tabular}
\end{center}
in this class, admits a Bautin bifurcation (see \Cref{rem:A-H} \eqref{item:bautin}).

We now argue that the bifurcations we have characterised so far are, indeed, the only possible generic bifurcations of positive equilibria which can occur in planar, quadratic mass-action networks with up to four reactions, and molecularity three or less. First, trivially, no planar network can admit a bifurcation which requires a center manifold of dimension three or higher, ruling out fold-Hopf and Hopf-Hopf bifurcations \cite[Sections 8.4 and 8.5]{kuznetsov:2023} in these networks. Second, in a $(2,4,2)$ network, bifurcations of codimension greater than $2$ are ruled out by the following theorem. 
\begin{theorem}
    \label{thm:bifcodim}
An $(n,m,n)$ mass-action network admits bifurcations of codimension at most $m-2$ in the sense that no bifurcation of codimension greater than $m-2$ can be unfolded by the rate constants of the network.
\end{theorem}
\begin{proof}
This follows immediately when we apply the recoordinatisation discussed in \Cref{sec:coords} to such a network, to obtain a family of ODEs with $m-2$ parameters.
\end{proof}

In order to complete the bifurcation analysis of planar, quadratic, trimolecular mass-action networks with up to four reactions, we need to answer two questions: 
\begin{enumerate}
    \item Can cusp bifurcations occur in networks in this class?
    \item Can any bifurcations of codimension $3$ or higher occur in {\em rank-one} networks in this class? 
\end{enumerate}
The answer in both cases is no. We treat the question of rank-one networks first. Clearly the only possible bifurcations in rank-one networks are fold bifurcations, cusp bifurcations, and more degenerate fold-like bifurcations. However, as we saw earlier in the proof of \Cref{lem:quadratic_rank-one_fold}, in local coordinates on any stoichiometric class of a quadratic, rank-one network, we obtain a one-dimensional at-most-quadratic differential equation. Trivially, no more than two positive, nondegenerate equilibria are possible in such an ODE, ruling out cusp bifurcations and any other bifurcations leading to three or more nondegenerate equilibria. 

All that remains is to show that cusp bifurcations cannot occur amongst the {\em rank-two} networks examined so far. This follows from the following more general result which is again a consequence of viewing mass-action networks in the natural coordinates described in \Cref{sec:coords}.

\begin{theorem}
\label{thm:cusp}
An $(n,n+2,n)$ mass-action network forbids cusp bifurcations in the sense that no cusp bifurcation can be unfolded by the rate constants of the network. 
\end{theorem}
\begin{proof}
After the recoordinatisation described in \Cref{sec:coords}, an $(n,n+2,n)$ mass-action CRN defines a system of ODEs on $\mathbb{R}^n_+$ of the form
\[
\dot y = \gamma \circ f(y, \beta)
\]
where $\gamma \in \mathbb{R}^n_+$ and $\beta \in \mathbb{R}_+$. The result now follows from the following lemma.
\end{proof}

\begin{lemma}
Consider a family of ODEs on any open subset of $\mathbb{R}^n$ of the form
\[
\dot y = \gamma \circ f(y,\beta) =: \hat{f}(y,\beta,\gamma)\,,
\]
where $\gamma \in \mathbb{R}^n_+$ and $\beta \in \mathbb{R}$ are parameters. This family cannot have a cusp bifurcation, unfolded by its parameters.
\end{lemma}

\begin{proof} 
Let $J(y,\beta,\gamma) = \partial_y\hat{f}(y,\beta,\gamma)$ be the Jacobian matrix of $\hat{f}$ w.r.t.\ $y$, evaluated at $(y,\beta,\gamma)$, and let $B_{y,\beta,\gamma}(\cdot\,,\cdot)$ be the multilinear function representing the quadratic terms in the Taylor expansion of $\hat{f}$ w.r.t.\ $y$, evaluated at $(y,\beta,\gamma)$. Note that $B_{y,\beta,\gamma}(\cdot\,,\cdot) = \gamma \circ \gamma_0^{-1} \circ B_{y,\beta,\gamma_0}(\cdot\,,\cdot)$ and, in fact, $\gamma$ ``factors'' out of the multilinear functions representing derivatives of all orders in the Taylor expansion of $\hat{f}$.

Let us suppose that $(y_0,\beta_0,\gamma_0)$ is a potential cusp bifurcation point (not necessarily nondegenerate), i.e., following \cite[Subsection 5.4.1]{kuznetsov:2023}, $\hat{f}(y_0,\beta_0,\gamma_0) = 0$; $J(y_0,\beta_0,\gamma_0)$ has a simple zero eigenvalue; and $\langle p_0, B_{y_0,\beta_0,\gamma_0}(q_0,q_0)\rangle = 0$, where $p_0$ and $q_0$ are left and right eigenvectors of $J(y_0,\beta_0,\gamma_0)$, corresponding to the eigenvalue zero. For definiteness, we impose the normalisation conditions $|q_0|=1$ and $\langle p_0, q_0\rangle = 1$. As zero is a simple eigenvalue of $J(y_0,\beta_0,\gamma_0)$, we can smoothly continue these eigenvectors (with the same normalisation conditions) in a neighbourhood of $(y_0,\beta_0,\gamma_0)$. Write $p(y,\beta,\gamma)$ and $q(y,\beta,\gamma)$ for these continuations. Note that if we fix $(y,\beta) = (y_0, \beta_0)$ and vary only $\gamma$, then we see that $q(y_0,\beta_0,\gamma) = q_0$ and $p(y_0,\beta_0,\gamma) = (\gamma^{-1} \circ \gamma_0 \circ p_0)/C(\gamma)$, where $C(\gamma)$ is the normalisation constant chosen to ensure that $\langle p(y_0,\beta_0,\gamma), q_0\rangle = 1$.

Write
\[
F(y,\beta,\gamma):=\mathrm{det}(J(y,\beta,\gamma)), \quad G(y,\beta,\gamma) := \langle p(y,\beta,\gamma), B_{y,\beta,\gamma}(q(y,\beta,\gamma),q(y,\beta,\gamma))\rangle
\]
so that the basic conditions for cusp bifurcation become $\hat{f}=0, F=0$ and $G=0$. We have assumed that $\hat{f}(y_0,\beta_0,\gamma_0)=0, F(y_0,\beta_0,\gamma_0)=0$ and $G(y_0,\beta_0,\gamma_0)=0$. 
The bifurcation at $(y_0,\beta_0,\gamma_0)$ is unfolded transversely by the parameters only if the following $(n+2) \times (2n+1)$ Jacobian matrix is surjective:
\[
\left.\frac{\partial(\hat{f},F,G)}{\partial(y,\beta,\gamma)}\right|_{y_0,\beta_0,\gamma_0}:=\left.\left(\begin{array}{ccc}\partial_y\hat{f}(y,\beta,\gamma) & \partial_\beta\hat{f}(y,\beta,\gamma)&\partial_\gamma\hat{f}(y,\beta,\gamma)\\\partial_yF(y,\beta,\gamma) & \partial_\beta F(y,\beta,\gamma)&\partial_\gamma F(y,\beta,\gamma)\\\partial_yG(y,\beta,\gamma) & \partial_\beta G(y,\beta,\gamma)&\partial_\gamma G(y,\beta,\gamma)\end{array}\right)\right|_{y=y_0,\gamma=\gamma_0,\beta=\beta_0}\,.
\]
We claim, however, that, $\partial_\gamma\hat{f}(y_0,\beta_0,\gamma_0)$, $\partial_\gamma F(y_0,\beta_0,\gamma_0)$ and $\partial_\gamma G(y_0,\beta_0,\gamma_0)$ are all zero and so the above matrix has rank at most $n+1$ and surjectivity is impossible. 
\begin{enumerate}
\item $D_\gamma\hat{f}(y_0,\beta_0,\gamma_0) = f(y_0,\beta_0) = 0$, since $\hat{f}(y_0,\beta_0,\gamma_0) = 0$ implies $f(y_0,\beta_0)=0$.
\item $F(y,\beta,\gamma):=(\Pi_i\gamma_i)\mathrm{det}(\partial_yf(y,\beta))$, and so $F(y_0,\beta_0,\gamma_0) = 0$ implies $\mathrm{det}(\partial_yf(y_0,\beta_0)) = 0$. Consequently $\partial_\gamma F(y_0,\beta_0,\gamma_0) = 0$.
\item From above, $B_{y_0,\beta_0,\gamma}(\cdot\,,\cdot) = \gamma \circ \gamma_0^{-1} \circ B_{y_0,\beta_0,\gamma_0}(\cdot\,,\cdot)$, $q(y_0,\beta_0,\gamma) = q_0$, $p(y_0,\beta_0,\gamma) = (\gamma^{-1} \circ \gamma_0 \circ p_0)/C(\gamma)$, and so
\begin{eqnarray*}
G(y_0,\beta_0,\gamma) &=& \left\langle p(y_0,\beta_0,\gamma), B_{y_0,\beta_0,\gamma}(q(y_0,\beta_0,\gamma),q(y_0,\beta_0,\gamma))\right\rangle\\
&=& \left\langle \gamma^{-1} \circ \gamma_0 \circ p_0, \gamma \circ \gamma_0^{-1} \circ B_{y_0,\beta_0,\gamma_0}(q_0,q_0)\right\rangle/C(\gamma)\\
&=& \left\langle p_0, B_{y_0,\beta_0,\gamma_0}(q_0,q_0)\right\rangle/C(\gamma)\\
&=& G(y_0,\beta_0,\gamma_0)/C(\gamma)\\
&=& 0.
\end{eqnarray*}
Consequently, it is clear that $\partial_\gamma G(y_0,\beta_0,\gamma_0)=0$, as claimed. 
\end{enumerate}
\end{proof}

\begin{remark}
The fact that cusp bifurcations are ruled out in $(n,n+2,n)$ mass-action networks does not rule out the possibility of three or more positive nondegenerate equilibria in such networks, and indeed it is possible to find quadratic $(2,4,2)$ mass-action networks with three positive nondegenerate equilibria, provided we allow sufficiently high product molecularity. The following network is easily checked to be such an example.  
\begin{align*}
\X+\Y \to \X \to \0 \to 5\X+\Y \qquad 2\Y \to \X+4\Y
\end{align*}
\end{remark}

\section{Conclusions}
\label{sec:conclusions}

We have fully characterised all generic bifurcations of positive equilibria which can occur in planar, quadratic mass-action networks with up to four reactions and molecularity of three or less. Such networks admit two bifurcations of codimension one, namely fold and Andronov--Hopf bifurcations; and two bifurcations of codimension two, namely Bogdanov--Takens and Bautin bifurcations. They admit no other generic bifurcations of equilibria of any codimension. We have also discussed the occurrence of various degenerate bifurcations in this class of networks, including vertical Andronov--Hopf bifurcations, and a bifurcation we have termed a vertical Bogdanov--Takens bifurcation. It is interesting to note that these highly degenerate bifurcations occur fairly frequently in small mass-action networks.

We have shown that cusp bifurcation in planar, quadratic mass-action networks requires at least five reactions. One reason for an interest in cusp bifurcation is that it is relevant to the study of {\em multistability}: it gives us one way of finding networks admitting two positive, {\em asymptotically stable} equilibria. We plan to study quadratic $(2,5,2)$ networks admitting cusp bifurcations in future work. In fact, quadratic $(2,5,2)$ networks include the simplest {\em bimolecular} networks admitting cusp bifurcation.  

Apart from cusp bifurcations, in quadratic $(2,5,2)$ mass-action networks we expect to see networks admitting supercritical Bautin bifurcations where a stable limit cycle can coexist with a stable equilibrium. We also hope to find networks admitting bifurcations of codimension 3 (e.g.\ a degenerate Bogdanov--Takens bifurcation or a degenerate Bautin bifurcation, sometimes termed a Takens--Hopf bifurcation).

There are some open questions about the quadratic, trimolecular, $(2,4,2)$ mass-action networks we have studied in this paper. One question is whether any of them, other than the unique network which admits a Bautin bifurcation, admits more than one limit cycle. It is a nontrivial fact that multiple limit cycles cannot occur in a planar quadratic ODE with an invariant line \cite{coppel:1989}, and as a consequence of this result, multiple limit cycles are ruled out in many of the 198 quadratic, trimolecular $(2,4,2)$ networks which admit Andronov--Hopf bifurcation (in 106 networks at least one of the two coordinate axes is invariant). However, for the remaining 91 networks (those with no invariant axis and admitting no Bautin bifurcation), we cannot claim definitively that none admits multiple limit cycles. Another open question is whether homoclinic orbits can occur in any of the seven quadratic, trimolecular $(2,4,2)$ networks which admit both fold and Andronov--Hopf bifurcations, but where the bifurcation sets do not meet and no Bogdanov--Takens bifurcation occurs.

We have identified 10 bimolecular $(2,4,2)$ networks (those marked with an orange bullet in \Cref{tab:fold_bimolecular}) that admit multiple nonnegative asymptotically stable equilibria. It is an open problem whether there are any further bimolecular $(2,4,2)$ mass-action networks with the same property, and which are \emph{not} dynamically equivalent to any of the 10 we have found. Answering this question requires studying boundary equilibria.

\bibliographystyle{abbrv}
\bibliography{biblio}

\appendix
\section{Vertical Andronov--Hopf bifurcation}
\label{sec:A-H_vertical}

In \Cref{thm:A-H} we have identified 20 networks which admit a vertical Andronov--Hopf bifurcation (the networks in \Cref{tab:A-H} that are marked with \blue{$0$} or \gbllt \bbllt \obllt). In \Cref{tab:A-H_vertical} below we have listed the 20 networks along with the conditions on the rate constants under which the systems have a positive equilibrium that is a center. In \Cref{lem:A-H_vertical} below we show directly (i.e., without computation of focal values) that these 20 systems indeed have a center.

\begin{lemma} \label{lem:A-H_vertical}
Each network in \Cref{tab:A-H_vertical} has a center if the rate constants are as indicated.
\end{lemma}
\begin{proof}
It is straightforward to verify that each of the 20 systems has a positive equilibrium with a pair of purely imaginary eigenvalues whenever the conditions on rate constants in \Cref{tab:A-H_vertical} hold. Below we argue why the equilibrium is a center.
\begin{description}
\item[Networks 1--5] Restricting the mass-action differential equation to the positive quadrant and dividing the r.h.s.\ by $x$ results in a linear equation. An equilibrium of a planar linear system with a pair of purely imaginary eigenvalues is a center.
\item[Network 6] The line $\kappa_1 x = \kappa_2 y$ is invariant. After the linear coordinate change with the shear map $\begin{pmatrix} 1 & 0 \\ -\frac{\kappa_1}{\kappa_2} & 1 \end{pmatrix}$, the differential equation is a Lotka--Volterra equation. A positive equilibrium of a planar Lotka--Volterra system with a pair of purely imaginary eigenvalues is a center 
\cite[p.~213]{andronov:1973}.

\item[Networks 7--9 and 14--17] Each network has a unique positive equilibrium $(x^*,y^*)$ for all rate constants satisfying the conditions in \Cref{tab:A-H_vertical}. After the linear coordinate change with the shear map $\begin{pmatrix} 1 & -\frac{x^*}{y^*} \\ 0 & 1 \end{pmatrix}$, the differential equation is reversible w.r.t.\ the vertical axis (i.e., $\dot{x}=f(x,y)$, $\dot{y}=g(x,y)$ with $f$ being even in $x$ and $g$ being odd in $x$). For reversible systems, an equilibrium (on the line of reflection) with a pair of purely imaginary eigenvalues is a center.
\item[Networks 10--13] Each mass-action differential equation is a Lotka--Volterra equation.
\item[Network 18] By the coordinate change $x \mapsto x+x^*$ (where $(x^*,y^*)$ is the unique positive equilibrium for all rate constants satisfying the conditions in \Cref{tab:A-H_vertical}) we obtain a system which is reversible w.r.t.\ the vertical axis.
\item[Networks 19 and 20] Restricting the mass-action differential equation to the positive quadrant and dividing the r.h.s.\ by $x$ results in a Hamiltonian system. The constant of motion is $H(x,y)=\kappa_1 x y+\frac{\kappa_2}{2}y^2-\kappa_3 y-\kappa_4 \log x$ for network 19, and it is similar for network 20 (replace the coefficient $\frac{\kappa_2}{2}$ by $\kappa_2$).
\end{description}
\end{proof}

\begin{table}[!ht]
\centering
{\footnotesize
\begin{tabular}{r|*{4}{r@{\hspace{4pt}$\to$\hspace{4pt}}l}|ll}
 1 & $2\X$ & $3\X$ & $\X+\Y$ &  $\0$ & $2\X$ &  $\Y$         & $\X$ & $2\X$ &
 $\kappa_1 = \kappa_2 +2\kappa_3$ & $\kappa_2<\kappa_3$ \\
 2 & $2\X$ & $3\X$ & $\X+\Y$ &  $\0$ & $2\X$ & $2\Y$         & $\X$ & $2\X$ &
 $\kappa_1 = \kappa_2 +2\kappa_3$ & $\kappa_2<2\kappa_3$ \\
 3 & $2\X$ & $3\X$ & $\X+\Y$ &  $\0$ & $2\X$ & $3\Y$         & $\X$ & $2\X$ &
 $\kappa_1 = \kappa_2 +2\kappa_3$ & $\kappa_2<3\kappa_3$ \\
 4 & $2\X$ & $3\X$ & $\X+\Y$ &  $\0$ & $2\X$ & $\X+2\Y$ & $\X$ & $2\X$ &
 $\kappa_1 = \kappa_2 + \kappa_3$ & $\kappa_2<2\kappa_3$ \\
 5 & $2\X$ & $3\X$ & $\X+\Y$ &  $\0$ & $2\X$ & $2\X+\Y$ & $\X$ & $2\X$ &
 $\kappa_1 = \kappa_2$          & $\kappa_2<\kappa_3$   \\
 6 & $2\X$ & $3\X$ & $\X+\Y$ &  $\Y$ &  $\X$ & $2\X+\Y$ & $\Y$ &  $\0$ &
 $\kappa_1(\kappa_3+\kappa_4) = \kappa_2\kappa_3$ & \\
 7 & $2\X$ & $3\X$ & $\X+\Y$ &  $\Y$ &  $\X$ & $2\X+\Y$ & $\Y$ &  $\X$ &
 $2\kappa_1 = \kappa_2$ and $\kappa_3=\kappa_4$ & \\
 8 & $2\X$ & $3\X$ & $\X+\Y$ &  $\Y$ &  $\X$ & $2\X+\Y$ & $\Y$ & $2\X$ &
 $2\kappa_1 = \kappa_2$ and $\kappa_3=\kappa_4$ & \\
 9 & $2\X$ & $3\X$ & $\X+\Y$ &  $\Y$ &  $\X$ & $2\X+\Y$ & $\Y$ & $3\X$ &
 $2\kappa_1 = \kappa_2$ and $\kappa_3=\kappa_4$ & \\
10 & $2\X$ & $3\X$ & $\X+\Y$ & $2\Y$ & $2\Y$ & $\0$          & $\Y$ &  $\0$ &
$\kappa_2 = 2\kappa_3$ & $\kappa_1<\kappa_2$ \\
11 & $2\X$ & $3\X$ & $\X+\Y$ & $3\Y$ & $2\Y$ & $\0$          & $\Y$ &  $\0$ &
$\kappa_2 = 2\kappa_3$ & $\kappa_1<2\kappa_2$  \\
12 & $2\X$ & $3\X$ & $\X+\Y$ & $2\Y$ & $2\Y$ & $\0$          & $\X$ & $2\X$ &
$\kappa_1 = \kappa_2$ & $2\kappa_3<\kappa_2$ \\
13 & $2\X$ & $3\X$ & $\X+\Y$ & $3\Y$ & $2\Y$ & $\0$          & $\X$ & $2\X$ &
$\kappa_1 = 2\kappa_2$ & $2\kappa_3<\kappa_2$ \\
14 & $2\X$ & $3\X$ & $\X+\Y$ & $2\Y$ & $2\Y$ & $\0$          & $\0$ &  $\X$ &
$4\kappa_1\kappa_3=\kappa_2(\kappa_2+2\kappa_3)$ & $\kappa_2<\kappa_1$ \\
15 & $2\X$ & $3\X$ & $\X+\Y$ & $3\Y$ & $2\Y$ & $\0$          & $\0$ &  $\X$ &
$2\kappa_1\kappa_3=\kappa_2(\kappa_2+2\kappa_3)$ & $2\kappa_2<\kappa_1$ \\
16 & $2\X$ & $3\X$ & $\X+\Y$ & $2\Y$ & $2\Y$ & $\0$          & $\Y$ &  $\X+\Y$ &
$4\kappa_1\kappa_3=\kappa_2(\kappa_2+2\kappa_3)$ & $2\kappa_3<\kappa_2$ \\
17 & $2\X$ & $3\X$ & $\X+\Y$ & $3\Y$ & $2\Y$ & $\0$          & $\Y$ &  $\X+\Y$ &
$2\kappa_1\kappa_3=\kappa_2(\kappa_2+2\kappa_3)$ & $2\kappa_3<\kappa_2$ \\
18 & $2\X$ & $3\X$ & $\X+\Y$ &  $\X$ &  $\X$ & $\0$          & $\Y$ &  $\X+2\Y$ &
$2\kappa_1\kappa_4=\kappa_2\kappa_3$ & \\
19 & $2\X$ & $3\X$ & $\X+\Y$ & $2\X$ &  $\X$ & $\0$          & $\0$ &  $\Y$ &
$\kappa_1=\kappa_2$ & $4\kappa_1\kappa_4 < \kappa_3^2$ \\
20 & $2\X$ & $3\X$ & $\X+\Y$ & $3\X$ &  $\X$ & $\0$          & $\0$ &  $\Y$ &
$\kappa_1=\kappa_2$ & $8\kappa_1\kappa_4 < \kappa_3^2$ \\
\end{tabular}
}
\caption{The 20 dynamically nonequivalent, quadratic, trimolecular $(2,4,2)$ mass-action networks that admit a vertical Andronov--Hopf bifurcation. In networks 7--9, the Andronov--Hopf bifurcation can also be subcritical and supercritical, while in networks 1--6 and 10--20 the Andronov--Hopf bifurcation is always vertical. We have indicated at the end of each line where in parameter space the vertical Andronov--Hopf bifurcation occurs.}
\label{tab:A-H_vertical}
\end{table}

\end{document}